\documentclass[letterpaper,10pt,reqno,onefignum,onetabnum]{amsart}
\usepackage[english]{babel}
\usepackage{amsmath}
\usepackage{amsthm}
\usepackage{verbatim}
\usepackage{mathrsfs}
\usepackage{bm}
\usepackage{cite}
\usepackage[foot]{amsaddr}
\usepackage[dvipsnames]{xcolor}
\usepackage{hyperref}
\hypersetup{
	colorlinks = true,
	linkcolor = OliveGreen,
	anchorcolor = OliveGreen,
	citecolor = OliveGreen,
	filecolor = OliveGreen,
	urlcolor = OliveGreen
}
\usepackage{algorithm}% http://ctan.org/pkg/algorithms
\usepackage{algorithmic}% http://ctan.org/pkg/algorithms
\usepackage{float}
\usepackage{lipsum}
\usepackage{amsfonts}
\usepackage{amssymb}
\usepackage{graphicx}
\usepackage{hyperref}
\usepackage{textcomp}
\usepackage{epstopdf}
\usepackage{cases}
\usepackage{multirow}
\ifpdf
\DeclareGraphicsExtensions{.eps,.pdf,.png,.jpg}
\else
\DeclareGraphicsExtensions{.eps}
\fi
\usepackage{verbatim}
\usepackage{mathrsfs}
\usepackage{bm}

\newtheorem{theorem}{Theorem}[section]
\newtheorem{proposition}[theorem]{Proposition}

\newtheorem{problem}[theorem]{Problem}
\newtheorem{remark}[theorem]{Remark}
\newtheorem{definition}[theorem]{Definition}

\usepackage{lipsum}
\usepackage{amsfonts}
\usepackage{amssymb}
\usepackage{graphicx}
\usepackage{epstopdf}
\usepackage{cases}
\usepackage{multirow}
\usepackage{algorithmic}
\usepackage{tikz}
\usepackage{booktabs}%
\usetikzlibrary{matrix}
\usetikzlibrary{arrows}
\ifpdf
\DeclareGraphicsExtensions{.eps,.pdf,.png,.jpg}
\else
\DeclareGraphicsExtensions{.eps}
\fi
\usepackage{verbatim}
\usepackage{mathrsfs}
\usepackage{bm}
\usepackage{color}
\usepackage{caption}
\usepackage{subcaption}
\usepackage{enumitem}
\usepackage{geometry}
\newcommand{\R}{\mathbb{R}}
\newcommand{\RR}{\mathbb{R}}

\newcommand{\HH}{\mathcal H}
\newcommand{\Hi}{\mathcal H}
\newcommand{\N}{\mathbb N}

\newcommand{\RX}{\,]-\infty,+\infty]}

\newcommand{\RPP}{\ensuremath{\left]0,+\infty\right[}}

\newcommand{\RP}{\ensuremath{\left[0,+\infty\right[}}

\newcommand{\menge}[2]{\left\{{#1}~\middle|~{#2}\right\}} 
\newcommand{\scal}[2]{{\left\langle{#1}~\middle|~{#2}\right\rangle}}
\newcommand\norm[1]{\left\|#1\right\|}
\newcommand{\emp}{\ensuremath{{\varnothing}}}

\newcommand{\bigp}[1]{\left( #1 \right)}

\DeclareMathOperator{\dom}{dom}

\DeclareMathOperator{\Id}{Id}

\DeclareMathOperator{\sri}{sri}

\allowdisplaybreaks

\newcommand{\prox}{\ensuremath{\text{\rm prox}}}

\DeclareMathOperator*{\argmin}{arg\,min}

\title[]{Optimal Leveraging of
	Smoothness and Strong Convexity for Peaceman--Rachford Splitting}
\author{Luis Briceño-Arias$^\ddag$}
\author{Fernando Rold\'an$^{\dagger}$}
\address{$^\ddag$ Departamento de Matemática, Universidad Técnica Federico Santa María,
	Santiago, Chile. {\it 
		E-mail address: } 
	{\sf{luis.briceno@usm.com}}. }
\address{$^{\dagger}$Departamento de Ingeniería Matemática, Universidad de Concepción, Concepción, Chile. {\it 
	E-mail address:} 
	{\sf{fernandoroldan@udec.cl}}. }

\begin{document}
	\begin{abstract}
	In this paper, we introduce a simple methodology to leverage 
	strong convexity and smoothness in order to obtain an optimal 
	linear convergence rate for the Peaceman--Rachford splitting (PRS)
	scheme applied to optimization problems involving two smooth strongly convex 
	functions. The approach consists of adding and subtracting 
	suitable quadratic terms from one function to the other so as to 
	redistribute strong convexity in the primal formulation and 
	smoothness in the dual formulation. This yields an equivalent 
	modified optimization problem in which each term has adjustable 
	levels of strong convexity and smoothness.
	In this setting, the Peaceman--Rachford splitting method 
	converges linearly to the solution of the modified problem
	with a 
	convergence rate that can be optimized with respect to the 
	introduced parameters. Upon returning to the original formulation, 
	this procedure gives rise to a modified variant of PRS.
	The optimal linear rate established in this work is strictly better 
	than the best rates previously available in the general setting. The 
	practical performance of the method is illustrated through an 
	academic example and applications in image processing.
		\par
		\bigskip
		
		%62H35, 94A08,
		
	\end{abstract}
	
	\maketitle

\section{Introduction}
\label{sec:1}
This paper is devoted to the efficient resolution of the following convex optimization problem defined in a real Hilbert space $\HH$.
We denote by $\Gamma_0(\HH)$ the class
of lower semicontinuous convex proper functions from $\HH$ to $\RX$, given $h\in\Gamma_0(\HH)$,
$\partial h$ denotes its subdifferential, and $h$ is $\gamma$-strongly convex for some $\gamma\ge 0$ if $h-\gamma\|\cdot\|^2/2$ is convex.
\begin{problem}
	\label{prob:main}
	Let $\rho\ge0$ and $\mu\ge 0$, let $f\in\Gamma_0(\HH)$ and 
	$g\in\Gamma_0(\HH)$ be $\rho$ and $\mu$-strongly convex 
	functions, respectively, and suppose that $\partial f$ and $\partial g$
	are $\alpha$ and $\beta$-cocoercive\footnote{An operator $A\colon\HH\to 2^{\HH}$ is $\alpha$-cocoercive for some $\alpha\ge 0$ if, for every $(x,y)\in\HH\times\HH$ and $(u,v)\in Ax\times Ay$, we have $\scal{x-y}{u-v}\ge\alpha\|u-v\|^2$.} operators, respectively, for 
	some $\alpha\ge 
	0$ and $\beta\ge 0$. 
	%Suppose that $\alpha\rho<1$ and 
	%$\beta\mu<1$.
	The problem is to 
	\begin{equation}
		\label{e:main}
		\underset{x\in\HH}{{\rm minimize}}\quad f(x)+g(x),
	\end{equation} 
	under the assumption that its set of solutions $S$ is nonempty.
\end{problem}
Throughout this paper, $0$-strongly convex functions, or functions with $0$-cocoercive subdifferentials, are simply convex functions. This problem appears in several areas as variational mean field games \cite{Bena15,BAKS18,ESAIM1,SVVA1}, optimal transport \cite{Bena15,Carli24,papa13}, image processing \cite{Sigpro,SPL1,pasc21}, among others. 

Problem~\ref{prob:main} can be solved by the Douglas-Rachford splitting (DRS)
\cite{DRS}
or the Peaceman--Rachford splitting (PRS) \cite{PRS}, the latter under an additional strong convexity or cocoercivity assumption on $f$ or $g$ \cite{Lion79}. Given $\tau>0$ and $z_0\in\HH$, PRS iterates
\begin{equation}
	(\forall n\in\N)\quad z_{n+1}=R_{\tau f}R_{\tau g}z_n,
\end{equation}
where, for every $h\in \Gamma_0(\HH)$, $R_{h}=2\prox_{h}-\Id$ is the reflected 
proximity operator of $h$, $\Id$ stands for the identity operator, and $\prox_{h}\colon\HH\to\HH$ is the proximity operator of $h$, which associates to every $x\in\HH$ the unique solution to the lower semicontinuous strongly convex problem
\begin{equation}
	\label{e:proxh}
	\underset{y\in\HH}{{\rm minimize}}\quad h(y)+\frac{1}{2}\|y-x\|^2.
\end{equation}
In the general convex case, the operators $R_{\tau f}$ and $R_{\tau g}$ are nonexpansive, which is not enough to guarantee the weak convergence of a fixed point procedure of PRS.
Indeed, by taking $\HH=\R^2$ and $f$ to be the indicator function of the $x$-axis and $g=0$, it is clear that $R_{\tau g}=\Id$, $R_{\tau f}$ is the reflection with respect to the $x$-axis, and PRS oscillates reflecting on the axis if the starting point is not on it (see, e.g., \cite[Section~5.2]{Mont18}).

On the other hand, for some $\lambda\in\left]0,1\right[$, DRS iterates
\begin{equation}
	\label{e:DRS}
	(\forall n\in\N)\quad z_{n+1}=(1-\lambda)z_n+\lambda R_{\tau f}R_{\tau g}z_n,
\end{equation}
which is a Krasnosel'ski\u{\i}-Mann modification of PRS \cite{Ecks92,Kras55,Mann53} achieving global weak convergence without additional assumptions \cite{Svai11,bauschke2011:Convex_analysis}\footnote{The classical DRS proposed in \cite{DRS} is when $\lambda=1/2$.}. Note that PRS
can be seen as a particular instance of \eqref{e:DRS}
when $\lambda=1$.

In the case when $\rho>0$ and $\alpha>0$, $R_{\tau f}$
is a strict contraction \cite[Theorem~1]{Giselsson2017IEEE}, i.e., it satisfies 
\begin{equation}
	(\forall x\in\HH)(\forall y\in\HH)\quad \|R_{\tau f}x-R_{\tau f}y\|\le r(\tau)\|x-y\|,
\end{equation}
where
$$r(\tau)=\max\left\{\frac{\tau\alpha^{-1}-1}{\tau\alpha^{-1}+1},\frac{1-\tau\rho}{1+\tau\rho}\right\}\in\left]0,1\right[.$$
In this setting, PRS converges linearly with linear convergence rate $r(\tau)$, which is minimized when $\tau=\sqrt{\alpha/\rho}$, leading to 
the optimal linear convergence rate of $r^*=r(\sqrt{\alpha/\rho})
=\frac{1-\sqrt{\alpha\rho}}{1+\sqrt{\alpha\rho}}$. In the case when $\rho>0$ and $\beta>0$, i.e., when $f$ is strongly convex and $g$ is smooth, DRS converges linearly with an optimal rate $1/(1+\sqrt{\beta\rho})$, which is achieved for $\tau=\sqrt{\alpha/\rho}$ and $\lambda=\frac{1+\sqrt{\beta\rho}/2}{1+\sqrt{\beta\rho}}$ \cite[Theorem~5.6]{Giselsson2017}. To the best of our knowledge,
these are the best rates available in the literature for Problem~\ref{prob:main} under strong convexity and smoothness for PRS and DRS. In \cite{Davi17}, 
a review of the convergence rates for DRS and PRS are provided under various regularity assumptions, but
the new derived linear convergence results are not valid for PRS. Of course, there exist alternative algorithms for solving Problem~\ref{prob:main} under strong convexity and smoothness, as the gradient method or the proximal-gradient (or forward-backward) splitting, but their 
optimal rates and numerical performance are worse than that of PRS, as studied in \cite{Sigpro,SPL1}. 

In this paper, we establish a sharp linear convergence rate for a new variant of PRS by optimally leveraging the strong convexity and smoothness properties of Problem~\ref{prob:main} within a unified framework. The rate is obtained by adding and subtracting 
suitable quadratic terms from one function to the other so as to 
redistribute strong convexity in the primal formulation and 
smoothness in the dual formulation. We obtain an equivalent 
modified optimization problem in which each term has adjustable 
parameters of strong convexity and smoothness.
By optimizing these parameters, we provide an optimal rate for general values of strong convexity and smoothness parameters
in a novel algorithm obtained by applying PRS to the equivalent formulation. The optimal rate we obtain reduces to \cite[Theorem~1]{Giselsson2017IEEE} when $\mu=\beta=0$, and is strictly lower than the optimal rates in \cite[Theorem~5.6]{Giselsson2017} when $\mu=\alpha=0$ and in \cite[Corollary~4.1]{Briceno2025Fista} for FISTA under strong convexity.
When strong convexity and/or smoothness appear in both functions, our approach provides an optimal balance of these properties leading to the best linear convergence rate in the literature.
We finally provide some numerical experiments in an academic example and in two image processing problems to illustrate the advantages of our approach.

The paper is organized as follows. In Section~\ref{sec:2} we 
present our notation and preliminary results. Our main results are provided in 
Section~\ref{sec:mainr}. Numerical experiments are studied in Section~\ref{sec:4}. 

\section{Notation and preliminaries}
\label{sec:2}
Throughout this paper, $\mathcal{H}$ is a real Hilbert space 
endowed with the inner product $\scal{\cdot}{\cdot}$ and 
associated norm $\norm{\cdot}$. Given a nonempty closed convex set $C\subset\Hi$,  $\sri (C)$ denotes its strong relative interior.
Let $h: \Hi \to \RX$. The domain of $h$ is $\dom h = \menge{x \in 
	\Hi} {h(x) < +\infty}$ and $h$ is proper if $\dom h\neq \emp$. 
Denote by $\Gamma_0(\mathcal{H})$ the class of proper lower 
semicontinuous convex functions from 
$\Hi$ to $]-\infty, +\infty]$. 
Suppose that $h \in \Gamma_0(\Hi)$. 
The subdifferential of $h$ is the set-valued 
operator
\begin{align}
	\label{eq:subdifferential}
	\partial h: \Hi \to 2^{\Hi} : x \mapsto 
	\menge{u \in \Hi}{(\forall y \in \Hi) \quad \scal{y-x}{u} + h(x) \leq h(y)},
\end{align}  
which is maximally monotone, i.e., it is monotone
$$(\forall (x,y)\in(\dom\partial h)^2)(\forall 
(u,v)\in\partial h(x)\times\partial h(y))\quad 
\scal{x-y}{u-v}\ge0$$
and, if $B\colon\HH\to 2^{\HH}$ is a monotone operator such that ${\rm gra}\, \partial h\subset {\rm gra}\, B$, then $\partial h=B$,
where $\dom \partial h=\menge{x\in\Hi}{\partial h(x)\neq\emp}$
and ${\rm gra}\,B=\menge{(x,u)\in\HH\times\HH}{u\in Bx}$.
The Fenchel conjugate of $h$ is 
\begin{align}
	\label{eq:conjugate}
	h^*: \Hi \to \RX: u \mapsto \sup_{x \in \Hi} \bigp{\scal{x}{u} - h(x)}.
\end{align} 
We have $h^*\in\Gamma_0(\Hi)$, $h^{**} = h$, and $(\partial 
h)^{-1}=\partial h^*$. Set
$\rho\in\RP$. We say that $h$ is $\rho$-strongly convex if 
$h-\rho\|\cdot\|^2/2$ is 
convex. In particular, $h$ is $0$-strongly convex if and only if $h$ is 
convex. Hence, it follows from \cite[Example~22.4(iv)]{bauschke2011:Convex_analysis} that 
$h$ is $\rho$-strongly convex then 
\begin{equation}
	\label{e:strongconvpart}
	(\forall (x,y)\in(\dom\partial h)^2)(\forall 
	(u,v)\in\partial h(x)\times\partial h(y))\quad 
	\scal{x-y}{u-v}\ge\rho\|x-y\|^2.
\end{equation}
Given $\alpha\ge 0$, $\partial h$ is $\alpha$-cocoercive if 
\begin{equation}
	\label{e:coco}
	(\forall (x,y)\in(\dom\partial h)^2)(\forall 
	(u,v)\in\partial h(x)\times\partial h(y))\quad 
	\scal{x-y}{u-v}\ge\alpha\|u-v\|^2.    
\end{equation}
Note that, $h$ is convex if and only if $\partial h$ is $0$-cocoercive. Moreover, if $\alpha>0$ and $\dom\partial h=\HH$, the $\alpha$-cocoercivity of $\partial h$
implies the G\^ateaux-differentiability of $h$ \cite[Propositions~16.27 \& 17.31]{bauschke2011:Convex_analysis}.
In addition, it follows from Cauchy-Schwarz inequality, that $\nabla h$ is $\alpha^{-1}$-Lipschitz continuous, which implies the Fr\'echet differentiability of $h$ \cite[Corollary~17.42]{bauschke2011:Convex_analysis}. 
Actually, we have 
the following stronger result, which was first studied in \cite{BH77}.
\begin{theorem}{\cite[Theorem~18.15]{bauschke2011:Convex_analysis}}
	\label{t:BH}
	Let $h\colon\HH\to\RR$ be continuous and convex and let 
	$\alpha>0$. The following are equivalent:
	\begin{enumerate}
		\item 
		\label{t:BHi}
		$h$ is Fr\'echet differentiable and $\nabla h$ is $\alpha^{-1}$-Lipschitz 
		continuous.
		\item
		\label{t:BHii}
		$h$ is Fr\'echet differentiable and
		$$(\forall x\in\HH)(\forall y\in\HH)\quad \scal{x-y}{\nabla h(x)-\nabla 
			h(y)}\le \alpha^{-1}\|x-y\|^2.$$
		\item 
		\label{t:BHiii}
		$h$ is Fr\'echet differentiable and $\nabla h$ is $\alpha$-cocoercive.
		\item 
		\label{t:BHiv}
		$h^*$ is $\alpha$-strongly convex.
	\end{enumerate}
\end{theorem}

The \textit{proximity operator} of $h$ is
\begin{align}
	\label{eq:prox_def}
	\prox_h : \Hi \to \Hi : x \mapsto \argmin_{y \in \Hi} \left(h(y) + 
	\frac{1}{2} \norm{x-y}^2\right),
\end{align}
which is characterized by
\begin{equation}
	\label{eq:prox_char}
	(\forall x \in \Hi)(\forall p \in \Hi) \quad p = \prox_h x 
	\quad\Leftrightarrow\quad x-p \in \partial h(p)
\end{equation}
and satisfies 
\begin{align}
	\label{eq:moreau_decomp}
	(\forall\gamma \in\RPP)\quad  \prox_{\gamma h} =\Id - \gamma 
	\prox_{h^*/\gamma}\circ (\Id/\gamma),
\end{align}
where $\Id\colon\Hi\to\Hi$ denotes the identity operator. The 
\textit{reflected proximity operator} of $h$ is
$$
R_h=2\prox_h-\Id.
$$

In addition, we denote by ${\rm Fix}\, T=\menge{x\in\HH}{x=Tx}$ the set of fixed points of an operator $T\colon\HH\to\HH$.
For further background on convex analysis, the reader is referred 
to \cite{bauschke2011:Convex_analysis}. 
\section{Main results}
\label{sec:mainr}
Note that, in the context of Problem~\ref{prob:main}, it follows from Theorem~\ref{t:BH} and Cauchy-Schwarz inequality that 
$\alpha\rho\in[0,1]$ and $\beta\mu\in[0,1]$. In the case when
$\alpha\rho=1$ and $\dom\partial f=\HH$, it follows from \eqref{e:strongconvpart} and 
Theorem~\ref{t:BH}\eqref{t:BHii} that, for every $x\in\HH$,
$\scal{\nabla f(x)}{x}=\rho\|x\|^2+\scal{\nabla f(0)}{x}$. Therefore,
by integrating, we deduce
\begin{equation}
	\label{e:quadcase}
	f(x)=f(0)+\scal{\nabla f(0)}{x}+\frac{\rho}{2}\|x\|^2.
\end{equation}
A similar quadratic form for $g$ is obtained if $\beta\mu=1$.
These cases will be studied separately.

\begin{definition}
	\label{def:delta}
	Given $h\in\Gamma_0(\HH)$ and $\delta\in\RR$,
	we define $h_{\delta}=h+\frac{\delta}{2}\|\cdot\|^2$.
\end{definition}

\begin{proposition}
	\label{p:strcoco}
	Let $h\in\Gamma_0(\HH)$ be a $\rho$-strongly convex function
	such that $\partial h$ is $\alpha$-cocoercive, for some $\rho\ge 0$
	and $\alpha\ge 0$ satisfying $\alpha\rho\in\left[0,1\right]$. 
	Then, the following hold:
	\begin{enumerate}
		\item\label{p:strcoco0} 
		Suppose that $\alpha\rho=1$ and $\dom\partial h=\HH$, and let $\delta\in[-\rho,+\infty[$. Then,
		there exists $a\in\RR$ and $b\in\HH$ such that
		\begin{equation}
			(h_{\delta})^*=
			\begin{cases}
				\iota_{\{b\}}-a,&\text{if}\:\:\delta=-\rho;\\
				\frac{\|\cdot-b\|^2}{2(\rho+\delta)}-a,&\text{if}\:\:\delta>-\rho
			\end{cases}
		\end{equation}
		and, for every $\eta\in\RR$,
		\begin{equation}
			(((h_{\delta})^*)_{\eta})^*=
			\begin{cases}
				\scal{b}{\cdot}+a-\frac{\eta}{2}\|b\|^2,&\text{if}\:\:\delta=-\rho;\\
				\iota_{\{-\frac{b}{\rho+\delta}\}}+a-\frac{\|b\|^2}{2(\rho+\delta)},&\text{if}\:\:\delta>-\rho\:\:\text{and}\:\:
				\eta=-\frac{1}{\rho+\delta};\\
				\frac{\|\cdot+\frac{b}{\rho+\delta}\|^2}{2(\eta+\frac{1}{\rho+\delta})}
				+a-\frac{\|b\|^2}{2(\rho+\delta)},
				&\text{if}\:\:\delta>-\rho\:\:\text{and}\:\:
				\eta>-\frac{1}{\rho+\delta};\\
				-\infty, &\text{otherwise}.
			\end{cases}
		\end{equation}
		\item\label{p:strcoco2} 
		Suppose that either $\alpha=0$ or $\alpha>0$ and $\dom\partial h=\HH$. Suppose that $\alpha\rho<1$, let 
		$\delta\in[-\rho,+\infty[$, and let
		$\eta\in[-\frac{\alpha}{1+\alpha\delta},+\infty[$. Then, the following 
		hold: 
		\begin{enumerate}
			\item\label{p:strcoco21}  $(h_{\delta})^*\in\Gamma_0(\HH)$ and it is 
			$\frac{\alpha}{1+\alpha\delta}$-strongly convex.
			\item\label{p:strcoco22} Suppose that $\delta>-\rho$. Then 
			\begin{enumerate}
				\item\label{p:strcoco22i}  $\dom (h_{\delta})^*=\HH$,
				$(h_{\delta})^*$ is Fréchet differentiable and $\nabla (h_{\delta})^*$ is
				$(\rho+\delta)$-cocoercive.
				\item\label{p:strcoco22ii} 
				$(((h_{\delta})^*)_{\eta})^*\in\Gamma_0(\HH)$ and it is 
				$\frac{\rho+\delta}{1+(\rho+\delta)\eta}$
				-strongly 
				convex. 
			\end{enumerate}
			\item\label{p:strcoco23} Suppose that $\eta>-\frac{\alpha}{1+\alpha\delta}$. Then $\dom (((h_{\delta})^*)_{\eta})^*=\HH$,
			$(((h_{\delta})^*)_{\eta})^*$ is Fréchet differentiable and $\nabla (((h_{\delta})^*)_{\eta})^*$
			is $(\frac{\alpha}{1+\alpha\delta}+\eta)$-cocoercive.
		\end{enumerate}
		\item\label{p:strcoco3}  For every 
		$\tau\in\left]\max\{-\eta,0\},+\infty\right[$, suppose that 
		$\delta(\tau+\eta)>-1$. Then
		\begin{equation}\label{eq:pstrcocov1} 
			\prox_{\tau 
				(((h_{\delta})^*)_{\eta})^*}=\frac{\eta}{\tau+\eta}\Id+\frac{\tau}{\tau+\eta}
			\prox_{\frac{\tau+\eta}{1+\delta(\tau+\eta)}h}
			\circ\left(\frac{\Id}{1+\delta(\tau+\eta)}\right)
		\end{equation}
		and 
		\begin{equation}\label{eq:pstrcocov2} 
			R_{\tau 
				(((h_{\delta})^*)_{\eta})^*}=\frac{2\tau}{\tau+\eta}
			\prox_{\frac{\tau+\eta}{1+\delta(\tau+\eta)}h}
			\circ\left(\frac{\Id}{1+\delta(\tau+\eta)}\right)
			-\frac{\tau-\eta}{\tau+\eta}\Id.
		\end{equation}
	\end{enumerate}
\end{proposition}
\begin{proof}{Proof}
	\ref{p:strcoco0}: It follows from 
	\eqref{e:quadcase} that 
	$h_{\delta}=a+\scal{b}{\cdot}+\frac{\rho+\delta}{2}\|\cdot\|^2$ for $a=h(0)$ 
	and $b=\nabla h(0)$. The result follows from 
	\cite[Proposition~13.19 \& Proposition 
	13.23]{bauschke2011:Convex_analysis} and straightforward 
	computations.
	\ref{p:strcoco21}: Since $h\in\Gamma_0(\HH)$ is $\rho$-strongly convex, $\phi:=h-\frac{\rho}{2}\|\cdot\|^2$ is convex, thus, $h_{\delta}=\phi+(\rho+\delta)\|\cdot\|^2/2 \in \Gamma_0(\Hi)$. Now, if $\alpha=0$, the result is clear because $0$-strongly convex functions are convex.
	If $\alpha>0$ and $\dom\partial h=\HH$, it follows from Theorem~\ref{t:BH} that $h$ is Fr\'echet differentiable, $\nabla h$ is $\alpha^{-1}$-Lipschitz continuous and, therefore, $\nabla h_{\delta}$ is $(\alpha^{-1}+\delta)$-Lipschitz continuous.
	The result follows from Theorem~\ref{t:BH} now applied to $h_{\delta}$.
	\ref{p:strcoco22i}: If $\delta >-\rho$, $h_{\delta}$ is $(\rho+\delta)$-strongly convex and the result is a consequence of Theorem~\ref{t:BH}.
	\ref{p:strcoco22ii} : It follows from \ref{p:strcoco22i} and \ref{p:strcoco21} applied to $(h_{\delta})^*$.
	\ref{p:strcoco23} 
	It follows directly by combining \ref{p:strcoco21} and \ref{p:strcoco22i} with the function $(h_{\delta})^*$.
	\ref{p:strcoco3}: It follows from \eqref{eq:moreau_decomp} and 
	\cite[Proposition~24.8(i)]{bauschke2011:Convex_analysis} 
	that
	\begin{align}
		\prox_{\tau 
			(((h_{\delta})^*)_{\eta})^*}&=\Id-\tau\prox_{((h_{\delta})^*)_{\eta}/\tau}
		\circ\frac{\Id}{\tau}\nonumber\\
		&=\Id-\tau\prox_{(h_{\delta})^*/(\tau+\eta)}
		\circ\frac{\Id}{\tau+\eta}\nonumber\\
		&=\Id-\tau\left(\frac{\Id-\prox_{(\tau+\eta)h_{\delta}}}{\tau+\eta}\right)
		\nonumber\\
		&=\frac{\eta}{\tau+\eta}\Id+\frac{\tau}{\tau+\eta}
		\prox_{\frac{\tau+\eta}{1+\delta(\tau+\eta)}h}\circ\frac{\Id}{1+\delta(\tau+\eta)},
	\end{align} 
	and the result follows from $R_{\tau 
		(((h_{\delta})^*)_{\eta})^*}=2\prox_{\tau(((h_{\delta})^*)_{\eta})^*}-\Id$.
	 
\end{proof}
\begin{proposition}\label{p:solutionsets}
	In the context of Problem~\ref{prob:main}, 
	suppose that $0 \in\sri(\dom (f)-\dom (g))$,
	let 
	$\delta\in[-\rho,\mu]$, let
	$\eta\in\left[-\frac{\alpha}{1+\alpha\delta},\frac{\beta}{1-\beta\delta}\right]$, and consider the following optimization problem
	\begin{equation}
		\label{e:mainbd}
		\min_{x\in\HH}\big(((f_{\delta})^*)_{\eta}\big)^*(x)
		+\big(((g_{-\delta})^*)_{-\eta}\big)^*(x),
	\end{equation}
	where its solution set is denoted by $\widetilde{S}$. Then, the following hold:
	\begin{enumerate}
		\item \label{p:solutionsets1} If $\eta =0$, then $S=\widetilde{S}$.
		\item \label{p:solutionsets2} If $\eta>0$, then $S=\prox_{\frac{\eta}{1+\eta\delta}f}\left(\frac{1}{1+\eta\delta}\widetilde{S}\right)$.
		\item \label{p:solutionsets3} If $\eta<0$, then $S=\prox_{\frac{-\eta}{1+\eta\delta}g}\left(\frac{1}{1+\eta\delta}\widetilde{S}\right)$.
	\end{enumerate}
\end{proposition}
\begin{proof}{Proof}
	\ref{p:solutionsets1}: If $\eta = 0$, since $f_{\delta}$ and $g_{-\delta}$ are in $\Gamma_0(\HH)$, we have $((f_{\delta})^*)^*
	+((g_{-\delta})^*)^*=f_{\delta}+g_{-\delta}=f+g$ and the result follows.
	\ref{p:solutionsets2}: Let $x  \in S$. Since  $0 \in \sri (\dom (f) - \dom (g))$, we deduce from  \cite[Theorem~16.3 \& Corollary 16.48]{bauschke2011:Convex_analysis} that 
	\begin{align}
		x \in S \quad &\Leftrightarrow\quad  0 \in\partial f_{\delta}(x)+\partial g_{-\delta} (x)\nonumber\\
		&\Leftrightarrow\quad (\exists u \in \Hi) \quad  u \in \partial f_\delta (x) \quad \text{and} \quad  -u \in \partial g_{-\delta}(x)\label{e:auxsol}\\
		&\Leftrightarrow\quad (\exists u \in \Hi) \quad  x \in \partial (f_\delta)^* (u) \quad \text{and} \quad  x \in \partial (g_{-\delta})^*(-u)\nonumber\\
		&\Leftrightarrow\quad (\exists u \in \Hi) \quad  x+\eta u \in \partial ((f_\delta)^*)_{\eta} (u) \quad \text{and} \quad x+\eta u \in \partial ((g_{-\delta})^*)_{-\eta}(-u)\nonumber\\
		&\Leftrightarrow\quad (\exists u \in \Hi) \quad u \in \partial \big(((f_\delta)^*)_{\eta}\big)^* ( x+\eta u)\quad \text{and} \quad -u \in \partial \big(((g_{-\delta})^*)_{-\eta}\big)^*( x+\eta u),
	\end{align}
	which yields $z:=x+\eta u \in \widetilde{S}$ and, as a consequence, $\widetilde{S}\neq\varnothing$. Hence, it follows from \eqref{e:auxsol} and Definition~\ref{def:delta} that
	\begin{align}\label{eq:solutionsets21}
		\frac{z-x}{\eta} \in \partial f_{\delta} (x)=\partial f(x)+\delta x\quad   \Leftrightarrow \quad x =\prox_{\frac{\eta}{1+\eta\delta}f}\left(\frac{z}{1+\eta\delta}\right)
	\end{align}
	and we conclude that $S\subset\prox_{\frac{\eta}{1+\eta\delta}f}\left(\frac{1}{1+\eta\delta}\widetilde{S}\right)$. Conversely, let $z \in \widetilde{S}$ and $x=\prox_{\frac{\eta}{1+\eta\delta}f}\left(\frac{z}{1+\eta\delta}\right)$. It follows from \eqref{eq:solutionsets21} that 
	\begin{equation}
		\label{e:ufdel}
		u:= \frac{z-x}{\eta}\in\partial f_{\delta }(x).
	\end{equation} 
	Now, since $\eta >0\ge-\frac{\alpha}{1+\alpha\delta}$, Proposition~\ref{p:strcoco}.\ref{p:strcoco23} asserts that $\dom\big(((f_{\delta})^*)_{\eta}\big)^*=\Hi$ and \cite[Theorem~16.3 \& Corollary~16.48]{bauschke2011:Convex_analysis} yield
	\begin{align}
		z \in \widetilde{S}\quad &\Leftrightarrow\quad 0 \in \partial \big(((f_{\delta})^*)_{\eta}\big)^*(z) + \partial\big(((g_{-\delta})^*)_{-\eta}\big)^*(z)\nonumber\\
		&\Leftrightarrow\quad (\exists v \in \Hi) \quad -v \in \partial\big(((f_{\delta})^*)_{\eta}\big)^*(z) \quad \text{and} \quad  v \in \partial\big(((g_{-\delta})^*)_{-\eta}\big)^*(z)\nonumber\\
		&\Leftrightarrow\quad (\exists v \in \Hi) \quad z\in \partial (f_{\delta})^*(-v)-\eta v\quad \text{and} \quad z \in \partial(g_{-\delta})^*(v)-\eta v \nonumber\\
		&\Leftrightarrow\quad (\exists v \in \Hi) \quad -v\in \partial f_{\delta}(z+\eta v) \quad \text{and} \quad v \in \partial g_{-\delta}(z+\eta v)\label{e:auxsoldelta}\\
		&\Rightarrow\quad (\exists v \in \Hi) \quad 0\in \partial f(z+\eta v)+\partial g(z+\eta v)\nonumber\\
		&\Rightarrow\quad (\exists v \in \Hi) \quad 0\in \partial f(x+\eta (u+v))+\partial g(x+\eta (u+v)).
		\label{e:auxequiv}
	\end{align}
	Moreover, since $z=x+\eta u$ in view of \eqref{e:ufdel}, it follows from 
	\eqref{e:auxsoldelta} and the monotonicity of $\partial f_\delta$ that
	$$0\le \scal{u+v}{-\eta (u+v)}=-\eta\|u+ v\|^2,$$
	we conclude $u=-v$. Therefore, \eqref{e:auxequiv} implies $x \in S$. Finally, \ref{p:solutionsets3} is deduced analogously using $g_{-\delta}$ and $-\eta$ instead of $f_{\delta}$ and $\eta$.
\end{proof}
% \begin{remark}
	% Note that, the assumption in 
	% Proposition~\ref{p:probmodif}\eqref{p:probmodifiii} can be relaxed 
	% to $0\in\sri(\dom(f_{\delta})^*-\dom(g_{-\delta})^*)$ and 
	% $0\in\sri(\dom \varphi_{\delta,\eta}-\dom\psi_{\delta,\eta})$. Particularly, it hold if $\rho+\delta>0$ 
	% and $\frac{\alpha}{1+\alpha\delta}+\eta>0$.  Indeed,  
	% Proposition~\ref{p:strcoco}\eqref{p:strcocoii}\&\eqref{p:strcocoiv}
	% imply that $\dom\partial\varphi_{\delta,\eta}=\dom\partial(f_{\delta})^*=\HH$.
	% Then, since the cocoercivity property on the subdifferential 
	% implies its single-valuedness, functions $\varphi_{\delta,\eta}$ and 
	% $(f_{\delta})^*$ are differentiable in $\HH$, which yields
	% $\dom\varphi_{\delta,\eta}=\dom(f_{\delta})^*=\HH$.
	% \end{remark}
\begin{theorem}
	\label{t:main}
	In the context of Problem~\ref{prob:main},  suppose 
	that $\max\{\alpha \rho, \beta \mu\}<1$ and that $\min\{\rho+\mu,\alpha+\beta\}>0$. Let 
	$\delta\in[-\rho,\mu]$, 
	$\eta\in\left]-\frac{\alpha}{1+\alpha\delta},\frac{\beta}{1-\beta\delta}\right[$, and
	$\tau\in\left]|\eta|,+\infty\right[$ be such that 
	$\tau|\delta|<1+\delta\eta$. 
	Moreover, given $z_0\in\HH$, consider the recurrence
	\begin{equation}
		\label{e:DRthep}
		\begin{array}{l}
			\text{For}\:\:n=0,1,2,\ldots\\
			\left\lfloor
			\begin{array}{l}
				x_n=\prox_{\frac{\tau+\eta}{1+\delta(\tau+\eta)}f}
				\left(\dfrac{z_n}{1+\delta(\tau+\eta)}\right)\\[3mm]
				y_n=\frac{2\tau}{\tau+\eta} x_n-\frac{\tau-\eta}{\tau+\eta}z_n\\[2mm]
				p_n=\prox_{\frac{\tau-\eta}
					{1-\delta(\tau-\eta)}g}
				\left(\dfrac{y_n}{1-\delta(\tau-\eta)}
				\right)\\[3mm]
				z_{n+1}=z_n+\frac{2\tau}{\tau-\eta}(p_n-x_n).
			\end{array}
			\right.
		\end{array}
	\end{equation}
	Then, the following hold:
	\begin{enumerate}
		\item\label{t:main0}
		Set
		$\varphi_{\delta,\eta}:=\big(((f_{\delta})^*)_{\eta}\big)^*$ and
		$\psi_{\delta,\eta}:=\big(((g_{-\delta})^*)_{-\eta}\big)^*$.
		We have 
		\begin{equation}
			\prox_{\frac{\tau+\eta}{1+\delta(\tau+\eta)}f}\left(\frac{1}{{{1+\delta(\tau+\eta)}}}{\rm 
				Fix}(R_{\tau\psi_{\delta,\eta}}\circ R_{\tau\varphi_{\delta,\eta}})\right) = S
		\end{equation}
		and there exist a unique $z^*\in\HH$ such that ${\rm 
			Fix}(R_{\tau\psi_{\delta,\eta}}\circ R_{\tau\varphi_{\delta,\eta}})=\{z^*\}$.
		\item\label{t:maini}
		$\{z_n\}_{n\in\N}$ converges linearly to $z^*$, i.e.,
		$$(\forall n\in\N)\quad 
		\|z_{n+1}-z^*\|\le r(\tau,\eta,\delta)\|z_n-z^*\|,$$
		where $r(\tau,\eta,\delta)=r_1(\tau,\eta,\delta)
		r_2(\tau,\eta,\delta)<1$,
		\begin{equation}
			\label{e:r1}
			r_1(\tau,\eta,\delta)=
			\max\left\{\frac{(\tau-\eta)(1+\alpha\delta)-\alpha}
			{(\tau+\eta)(1+\alpha\delta)+\alpha},
			\frac{1-(\tau-\eta)(\rho+\delta)}
			{1+(\tau+\eta)(\rho+\delta)}\right\},
		\end{equation}
		and
		\begin{equation}
			\label{e:r2}
			r_2(\tau,\eta,\delta)=\max\left\{\frac{(\tau+\eta)(1-\beta\delta)-\beta}
			{(\tau-\eta)(1-\beta\delta)+\beta},
			\frac{1-(\tau+\eta)(\mu-\delta)}
			{1+(\tau-\eta)(\mu-\delta)}\right\}.
		\end{equation}
		\item\label{t:mainii} The optimal convergence rate is
		\begin{equation}
			\label{e:rstar}
			r^*=\frac{\sqrt{(1+\beta\rho)(1+\alpha\mu)}-
				\sqrt{(\alpha+\beta)(\rho+\mu)}}
			{\sqrt{(1+\beta\rho)(1+\alpha\mu)}+\sqrt{(\alpha+\beta)(\rho+\mu)}},
		\end{equation}
		which is achieved by taking
		\begin{equation}\label{eq:deltatauop}
			(\forall \delta\in[-\rho,\mu])\quad
			\begin{cases}
				\eta(\delta)=\dfrac{\beta\rho-\alpha\mu+\delta(\alpha(1+\beta\mu)
					+\beta(1+\alpha\rho))}{(\rho+\delta)(\mu-\delta)(\alpha+\beta)
					+(1+\alpha\delta)(1-\beta\delta)(\rho+\mu)}\\[5mm]
				\tau(\delta)=\dfrac{\sqrt{(\alpha+\beta)(\rho+\mu)(1+\alpha\mu)(1+\beta\rho)}}
				{(\rho+\delta)(\mu-\delta)(\alpha+\beta)
					+(1+\alpha\delta)(1-\beta\delta)(\rho+\mu)}.
			\end{cases} 
		\end{equation}
		\item\label{t:mainiii} The optimal rate $r^*$ in \eqref{e:rstar} is 
		tight, by taking 
		$\HH=\RR^2$,
		\begin{equation}
			\label{e:deffgtight}
			f\colon (x_1,x_2)\mapsto 
			\frac{\rho}{2}x_1^2+\frac{1}{2\alpha}x_2^2,\quad\text{and}\quad 
			g\colon (x_1,x_2)\mapsto 
			\frac{\mu}{2}x_1^2+\frac{1}{2\beta}x_2^2.
		\end{equation}
	\end{enumerate}
\end{theorem}

\begin{proof}{Proof}
	
	\ref{t:main0}:
	We assert that 
	\begin{equation}
		\label{e:Rcontractive}
		R_{\tau\varphi_{\delta,\eta}}\circ R_{\tau\psi_{\delta,\eta}}\quad\text{is $r(\tau,\eta,\delta)$-Lipschitz continuous with\: $r(\tau,\eta,\delta)<1$}.
	\end{equation}
	Hence, the Banach--Picard Theorem \cite[Theorem~1.50]{bauschke2011:Convex_analysis} yields ${\rm 
		Fix}(R_{\tau\psi_{\delta,\eta}}\circ R_{\tau\varphi_{\delta,\eta}})=\{z^*\}$ for some $z^*\in\HH$.
	Indeed, first suppose that $\delta\in]-\rho,\mu[$.
	Since $\rho+\mu>0$, it follows from Proposition~\ref{p:strcoco}\eqref{p:strcoco22} that $\varphi_{\delta,\eta}$ is $\frac{\rho+\delta}{1+(\rho+\delta)\eta}$-strongly convex and 
	$\psi_{\delta,\eta}$ is $\frac{\mu-\delta}{1-(\mu-\delta)\eta}$-strongly convex. Moreover, since 
	$\eta\in\left]-\frac{\alpha}{1+\alpha\delta},\frac{\beta}{1-\beta\delta}\right[$ and $\alpha+\beta>0$,
	Proposition~\ref{p:strcoco}\eqref{p:strcoco23} asserts that 
	$\varphi_{\delta,\eta}$ and $\psi_{\delta,\eta}$ are Fréchet differentiable with full domain and their gradients are $(\frac{\alpha}{1+\alpha\delta}+\eta)$- and $(\frac{\beta}{1-\beta\delta}-\eta)$-cocoercive, respectively. Therefore, since \eqref{e:r1} and \eqref{e:r2} can be written equivalently as
	\begin{multline}
		r_1(\tau,\eta,\delta)=\max\left\{\frac{\frac{\tau}{
				\frac{\alpha}{1+\alpha\delta}+\eta}-1}
		{\frac{\tau}{\frac{\alpha}{1+\alpha\delta}+\eta}+1},
		\frac{1-\frac{\tau(\rho+\delta)}{1+(\rho+\delta)\eta}}
		{1+\frac{\tau(\rho+\delta)}{1+(\rho+\delta)\eta}}\right\}\quad\text{and}\\
		r_2(\tau,\eta,\delta)=\max\left\{\frac{\frac{\tau}{
				\frac{\beta}{1-\beta\delta}-\eta}-1}
		{\frac{\tau}{
				\frac{\beta}{1-\beta\delta}-\eta}
			+1},
		\frac{1-\frac{\tau(\mu-\delta)}{1-(\mu-\delta)\eta}}
		{1+\frac{\tau(\mu-\delta)}{1-(\mu-\delta)\eta}}\right\},
	\end{multline}
	it follows from \cite[Theorem~1]{Giselsson2017IEEE} that $R_{\tau\varphi_{\delta,\eta}}$ and $R_{\tau\psi_{\delta,\eta}}$ are $r_1(\tau,\eta,\delta)$- and $r_2(\tau,\eta,\delta)$-Lipschitz continuous, with 
	$r_1(\tau,\eta,\delta)<1$ and $r_2(\tau,\eta,\delta)<1$ and \eqref{e:Rcontractive} holds. 
	In the case when $\delta=-\rho$, $R_{\tau\varphi_{\delta,\eta}}$ is merely nonexpansive in view of Proposition~\ref{p:strcoco}\eqref{p:strcoco2} and 
	\cite[Corollary~23.11]{bauschke2011:Convex_analysis}. However, as before,
	$R_{\tau\psi_{\delta,\eta}}$ is $r_2(\tau,\eta,\delta)$-Lipschitz continuous with $r_2(\tau,\eta,\delta)<1$, leading to the same conclusion. The case $\delta=\mu$ is analogous.
	
	Now set $x^* = \prox_{\frac{\tau+\eta}{1+\delta(\tau+\eta)}f}\left(\frac{z^*}{1+\delta(\tau+\eta)}\right)$, $y^* = R_{\tau\varphi_{\delta,\eta}} z^*$, and $p^* = \prox_{\frac{\tau-\eta}{1-\delta(\tau-\eta)}g}\left(\frac{y^*}{1-\delta(\tau-\eta)}\right)$. Then, it follows from Proposition~\ref{p:strcoco}\eqref{p:strcoco3} that
	\begin{equation}
		\label{e:auxystar}
		y^*=\frac{2\tau}{\tau+\eta} x^* - \frac{\tau-\eta}{\tau+\eta} z^*, 
	\end{equation}
	which yields
	\begin{align*}
		z^* = R_{\tau\psi_{\delta,\eta}}\circ R_{\tau\varphi_{\delta,\eta}} z^* \quad&\Leftrightarrow\quad z^* = \frac{2\tau}{\tau-\eta} p^* - \frac{\tau+\eta}{\tau-\eta} y^*\\
		&\Leftrightarrow\quad z^* = \frac{2\tau}{\tau-\eta} p^* - \frac{\tau+\eta}{\tau-\eta} \left( \frac{2\tau}{\tau+\eta} x^* - \frac{\tau-\eta}{\tau+\eta} z^* \right)\\
		& \Leftrightarrow\quad x^* = p^*.
	\end{align*}
	Hence, since
	\begin{equation*}
		\frac{z^*-(1+\delta(\tau+\eta))x^*}{\tau+\eta}\in \partial f (x^*) \quad \textnormal{ and } \quad \frac{y^*-(1-\delta(\tau-\eta))x^*}{\tau-\eta}\in \partial g (p^*)=\partial g (x^*)
	\end{equation*}
	and, in view of \eqref{e:auxystar},
	\begin{align*}
		\frac{z^*-(1+\delta(\tau+\eta))x^*}{\tau+\eta} + \frac{y^*-(1-\delta(\tau-\eta))x^*}{\tau-\eta}&=\frac{1}{\tau-\eta}\left(\frac{\tau-\eta}{\tau+\eta}z^*+ y^*  -\frac{2\tau}{\tau+\eta} x^*\right)
		= 0.
	\end{align*}
	Therefore, we conclude that $0 \in \partial f (x^*)+\partial g (x^*)$ and $x^*\in S$ in view of \cite[Theorem~16.3]{bauschke2011:Convex_analysis}. The result follows from the strict convexity of $f+g$ and \cite[Corollary~11.9]{bauschke2011:Convex_analysis}.
	
	% Suppose now that $0 \in \sri (\dom(f)-\dom(g))$. Then, if $x^* \in S$ we have $0\in\partial f(x^*)+\partial g(x^*)$ \cite[Corollary 27.3]{bauschke2011:Convex_analysis}. Hence, there exist $z^* \in \HH$ such that $\frac{z^*-(1+\delta(\tau+\eta))x^*}{\tau+\eta}\in \partial f (x^*)$ and $\frac{(1+\delta(\tau+\eta))x^*-z^*}{\tau+\eta}\in \partial g (x^*) $. Note that
	% \begin{align}\label{eq:Ssub1}
		% \frac{z^*-(1+\delta(\tau+\eta))x^*}{\tau+\eta}\in \partial f (x^*)  \Leftrightarrow x^* = \prox_{\frac{\tau+\eta}{1+\delta(\tau+\eta)}f}\left(\frac{z^*}{1+\delta(\tau+\eta)}\right). 
		% \end{align}
	% Moreover, by direct calculations, it follows from \eqref{eq:Ssub1} and Proposition~\ref{p:strcoco}.\eqref{p:strcoco3} that
	% \begin{align*}
		% \frac{(1+\delta(\tau+\eta))x^*-z^*}{\tau+\eta}\in \partial g (x^*)  \Leftrightarrow x^* &= \prox_{\frac{\tau-\eta}{1-\delta(\tau-\eta)}g}\left(\frac{1}{1-\delta(\tau-\eta)}\left( \frac{2\tau x^* -(\tau-\eta)z^*}{\tau+\eta}\right)\right) \\
		%  &= \prox_{\frac{\tau-\eta}{1-\delta(\tau-\eta)}g}\left(\frac{R_{\tau \varphi_{\delta,\eta}}z^*}{1-\delta(\tau-\eta)}\right) . 
		% \end{align*}
	% Therefore, invoking again Proposition~\ref{p:strcoco}.\eqref{p:strcoco3}, we obtain
	% \begin{equation*}
		%     R_{\tau \psi_{\delta,\eta}}(R_{\tau \varphi_{\delta,\eta}}z^*) = \frac{2\tau}{\tau-\eta}\prox_{\frac{\tau-\eta}{1-\delta(\tau-\eta)}g}\left(\frac{R_{\tau \varphi_{\delta,\eta}}z^*}{1-\delta(\tau-\eta)}\right)-\frac{\tau+\eta}{\tau-\eta}R_{\tau \varphi_{\delta,\eta}}z^*=z^*.
		%     \end{equation*}
	% This completes the proof. 

	\ref{t:maini}:
	It follows from 
	Proposition~\ref{p:strcoco}\eqref{p:strcoco3} that 
	\begin{equation}
		R_{\tau 
			\varphi_{\delta,\eta}}=
		\frac{2\tau}{\tau+\eta}\prox_{\frac{\tau+\eta}{1+\delta(\tau+\eta)}f}
		\circ\left(\frac{\Id}{1+\delta(\tau+\eta)}\right)-\frac{\tau-\eta}{\tau+\eta}\Id
	\end{equation}
	and 
	\begin{equation}
		R_{\tau 
			\psi_{\delta,\eta}}=
		\frac{2\tau}{\tau-\eta}\prox_{\frac{\tau-\eta}{1-\delta(\tau-\eta)}g}
		\circ\left(\frac{\Id}{1-\delta(\tau-\eta)}\right)-\frac{\tau+\eta}{\tau-\eta}\Id
	\end{equation}
	Hence, it follows from \eqref{e:DRthep} that, for every $n\in\N$,
	$y_n=
	R_{\tau \varphi_{\delta,\eta}}z_n$ and
	\begin{align}
		\label{e:PRS}
		(\forall n\in\N)\quad z_{n+1}
		&=z_n+\frac{2\tau}{\tau-\eta}\left[\prox_{\frac{\tau-\eta}
			{1-\delta(\tau-\eta)}g}
		\left(\frac{y_n}{1-\delta(\tau-\eta)}
		\right)-\prox_{\frac{\tau+\eta}{1+\delta(\tau+\eta)}f}
		\left(\frac{z_n}{1+\delta(\tau+\eta)}\right)\right]\nonumber\\
		&=\frac{2\tau}{\tau-\eta}\prox_{\frac{\tau-\eta}
			{1-\delta(\tau-\eta)}g}\left(\frac{y_n}{1-\delta(\tau-\eta)}
		\right)-\frac{\tau+\eta}{\tau-\eta}y_n\nonumber\\
		&=R_{\tau \psi_{\delta,\eta}}R_{\tau \varphi_{\delta,\eta}}z_n.
	\end{align}
	Then, the linear convergence follows from 
	\eqref{e:Rcontractive}.
	
	\ref{t:mainii}: Note that, given $\delta\in]-\rho,\mu[$ and 
	$\eta\in\left]-\frac{\alpha}{1+\alpha\delta},\frac{\beta}
	{1-\beta\delta}\right[$,
	by defining 
	\begin{equation}
		\phi_1\colon\tau\mapsto 
		\frac{(\tau-\eta)(1+\alpha\delta)-\alpha}
		{(\tau+\eta)(1+\alpha\delta)+\alpha}\quad\text{and}
		\quad\phi_2\colon\tau\mapsto 
		\frac{1-(\tau-\eta)(\rho+\delta)}
		{1+(\tau+\eta)(\rho+\delta)},
	\end{equation}
	which are strictly increasing and strictly decreasing, respectively,
	$r_1(\cdot,\eta,\delta)$ is minimized at 
	$\tau_1(\eta,\delta)>0$ 
	satisfying
	$\phi_1(\tau_1(\eta,\delta))=\phi_2(\tau_1(\eta,\delta))$ or, equivalently,
	\begin{align}
		\label{e:tauhat}
		\tau_1(\eta,\delta)=\sqrt{\frac{\alpha}{1+\alpha\delta}+\eta}
		\sqrt{\frac{1}{\rho+\delta}+\eta}.
	\end{align}
	We deduce 
	\begin{equation}
		\label{e:r_1hat}
		r_1(\tau_1(\eta,\delta),\eta,\delta)=
		\frac{\sqrt{\frac{\frac{1}{\rho+\delta}+\eta}
				{\frac{\alpha}{1+\alpha\delta}+\eta}}-1}{\sqrt{\frac{\frac{1}
					{\rho+\delta}+\eta}
				{\frac{\alpha}{1+\alpha\delta}+\eta}}+1}.
	\end{equation}
	Analogously, $r_2(\cdot,\eta,\delta)$ is minimized at 
	$\tau_2(\eta,\delta)>0$ satisfying 
	\begin{align}
		\label{e:taustar}
		\tau_2(\eta,\delta)=\sqrt{\frac{\beta}{1-\beta\delta}-\eta}
		\sqrt{\frac{1}{\mu-\delta}-\eta}
	\end{align}
	and
	\begin{equation}
		r_2(\tau_2(\eta,\delta),\eta,\delta)=
		\frac{\sqrt{\frac{\frac{1}{\mu-\delta}-\eta}
				{\frac{\beta}{1-\beta\delta}-\eta}}-1}{\sqrt{\frac{\frac{1}
					{\mu-\delta}-\eta}
				{\frac{\beta}{1-\beta\delta}-\eta}}+1}.
	\end{equation}
	Therefore, $r(\cdot,\eta,\delta)$ is minimized at 
	$\tau_1(\eta,\delta)=\tau_2(\eta,\delta)$, which is equivalent 
	to
	\begin{equation}
		\label{e:eta}
		\eta=
		\frac{\frac{\beta}{(1-\beta\delta)(\mu-\delta)}-\frac{\alpha}
			{(1+\alpha\delta)(\rho+\delta)}}{\frac{\alpha}{1+\alpha\delta}+\frac{1}{\rho+\delta}+
			\frac{\beta}{1-\beta\delta}+\frac{1}{\mu-\delta}}=:\eta(\delta)
	\end{equation}
	and 
	\begin{equation}
		\label{e:red}
		r(\tau_1(\eta,\delta),\eta,\delta)=
		\left(\frac{\sqrt{\frac{\frac{1}{\rho+\delta}+\eta}
				{\frac{\alpha}{1+\alpha\delta}+\eta}}-1}{\sqrt{\frac{\frac{1}
					{\rho+\delta}+\eta}
				{\frac{\alpha}{1+\alpha\delta}+\eta}}+1}\right)\left(\frac{\sqrt{\frac{\frac{1}{\mu-\delta}-\eta}
				{\frac{\beta}{1-\beta\delta}-\eta}}-1}{\sqrt{\frac{\frac{1}
					{\mu-\delta}-\eta}
				{\frac{\beta}{1-\beta\delta}-\eta}}+1}\right).
	\end{equation}
	By replacing \eqref{e:eta} in \eqref{e:red}, we obtain
	\begin{equation}
		\label{e:rd}
		r(\delta)=
		\left(
		\frac{\sqrt{\frac{(1+\beta\rho)(\rho+\mu)}{(1+\alpha\mu)(\alpha+\beta)}}
			\left(\frac{1+\alpha\delta}{\rho+\delta}\right)-1}
		{\sqrt{\frac{(1+\beta\rho)(\rho+\mu)}{(1+\alpha\mu)(\alpha+\beta)}}
			\left(\frac{1+\alpha\delta}{\rho+\delta}\right)+1}
		\right)
		\left(
		\frac{\sqrt{\frac{(1+\alpha\mu)(\rho+\mu)}{(1+\beta\rho)(\alpha+\beta)}}
			\left(\frac{1-\beta\delta}{\mu-\delta}\right)-1}
		{\sqrt{\frac{(1+\alpha\mu)(\rho+\mu)}{(1+\beta\rho)(\alpha+\beta)}}
			\left(\frac{1-\beta\delta}{\mu-\delta}\right)+1}
		\right),
	\end{equation} 
	which turns out to be constant (see Appendix), more precisely,
	\begin{equation}
		(\forall \delta\in]-\rho,\mu[)\quad 
		r(\delta)=r^*=\frac{\sqrt{(1+\beta\rho)(1+\alpha\mu)}-
			\sqrt{(\alpha+\beta)(\rho+\mu)}}
		{\sqrt{(1+\beta\rho)(1+\alpha\mu)}+\sqrt{(\alpha+\beta)(\rho+\mu)}}.
	\end{equation}
	Moreover, by using \eqref{e:eta} in \eqref{e:taustar}, we deduce
	$\tau_2(\eta(\delta),\delta)=\tau(\delta)$ defined in \eqref{eq:deltatauop}.
	
	On the other hand, if $\delta=\mu$, since $r_2(\tau,\eta,\mu)=1$, the 
	argument is simpler and leads to the same constant. In this case it is 
	only necessary to minimize $r_1$, which leads to an optimal 
	$\tau_1(\eta,\mu)$ defined in \eqref{e:tauhat} and
	an optimal value $r_1(\tau_1(\eta,\mu),\eta,\mu)$ defined in 
	\eqref{e:r_1hat}. Since $\eta\mapsto 
	r_1(\tau_1(\eta,\mu),\eta,\mu)$ is decreasing,
	its minimal value is obtained when 
	$\eta_1=\frac{\beta}{1-\beta\mu}$, which yields
	$r^*=r_1(\tau_1(\eta_1,\mu),\eta_1,\mu)$
	and that 
	\begin{equation}
		\tau_1(\eta_1,\mu)=\frac{1}{1-\beta\mu}
		\sqrt{\frac{(1+\beta\rho)(\alpha+\beta)}{(1+\alpha\mu)(\rho+\mu)}}=\tau(\mu).
	\end{equation}
	The case $\delta=-\rho$ leads to $r_1(\tau,\eta,\mu)=1$ and the 
	argument is analogous.
	
	\ref{t:mainiii}: It follows from \eqref{e:deffgtight} and 
	straightforward computations that
	\begin{equation}
		(\forall \gamma>0)\quad 
		\begin{cases}
			\prox_{\gamma f}\colon (x_1,x_2)\mapsto 
			\left(\frac{x_1}{1+\gamma\rho},\frac{x_2}{1+\gamma/\alpha}
			\right),\\
			\prox_{\gamma g}\colon 
			(x_1,x_2)\mapsto 
			\left(\frac{x_1}{1+\gamma\mu},\frac{x_2}{1+\gamma/\beta}
			\right).
		\end{cases}
	\end{equation}
	Then, let $\delta\in[-\rho,\mu]$ and let 
	$\eta=\eta(\delta)$ and
	$\tau=\tau(\delta)$ defined in \eqref{eq:deltatauop},
	it follows from Proposition~\ref{p:strcoco}\eqref{p:strcoco3} that 
	\begin{equation}
		\begin{cases}
			R_{\tau \varphi_{\delta,\eta}}\colon (x_1,x_2)\mapsto 
			\left(\left(\frac{1+(\eta-\tau)(\rho+\delta)}{1+(\eta+\tau)(\rho+\delta)}
			\right)x_1,-\left(\frac{(\tau-\eta)(1+\alpha\delta)-\alpha}
			{(\tau+\eta)(1+\alpha\delta)+\alpha}\right)x_2\right),\\[3mm]
			R_{\tau \psi_{\delta,\eta}}\colon 
			(x_1,x_2)\mapsto 
			\left(\left(\frac{1-(\tau+\eta)(\mu-\delta)}{1+(\tau-\eta)(\mu-\delta)}
			\right)x_1,-\left(\frac{(\tau+\eta)(1-\beta\delta)-\beta}
			{(\tau-\eta)(1-\beta\delta)+\beta}\right)x_2\right).
		\end{cases}
	\end{equation}
	Hence, by recalling that \eqref{e:DRthep} is equivalent to 
	\eqref{e:PRS} and that $\tau=\tau_1(\eta,\delta)$ is chosen for guaranteeing that $\phi_1( \tau_1(\eta,\delta))=\phi_2(\tau_1(\eta,\delta))$ and $\eta$ for obtaining $\tau=\tau_1(\eta,\delta)=\tau_2(\eta,\delta)$ (see \eqref{e:tauhat} and \eqref{e:taustar}), we deduce 
	\begin{equation}
		(\forall n\in\N)\quad 
		\begin{cases}
			z_{n+1}^1=\left(\frac{1-(\tau-\eta)(\rho+\delta)}{1+(\tau+\eta)(\rho+\delta)}
			\right)\left(\frac{1-(\tau+\eta)(\mu-\delta)}{1+(\tau-\eta)(\mu-\delta)}
			\right)z_n^1=r^*z_{n}^1\\
			z_{n+1}^2=\left(\frac{(\tau-\eta)(1+\alpha\delta)-\alpha}
			{(\tau+\eta)(1+\alpha\delta)+\alpha}\right)\left(\frac{(\tau+\eta)(1-\beta\delta)-\beta}
			{(\tau-\eta)(1-\beta\delta)+\beta}\right)z_n^2=r^*z_{n}^2.
		\end{cases}
	\end{equation}
	Moreover, since $z^*=(0,0)$, we deduce that 
	$\|z_{n+1}-z^*\|=r^*\|z_n-z^*\|$
	and the result follows.
	 
\end{proof}

\begin{remark}\label{rem:rates}
	\begin{enumerate}
		\item\label{rem:rates1} Note that the optimal rate obtained in Theorem~\ref{t:main} 
		is lower than the two possible rates derived from 
		\cite[Proposition~3]{Giselsson2017IEEE}. Indeed,
		noting that $x\mapsto \frac{1-x}{1+x}$ is strictly decreasing
		and 
		$\frac{(\alpha+\beta)(\rho+\mu)}{(1+\alpha\mu)(1+\beta\rho)}\ge 
		\max\{\alpha\rho,\beta\mu\}$, we have
		\begin{equation}
			r^*=\frac{1-\sqrt{\frac{(\alpha+\beta)(\rho+\mu)}{(1+\alpha\mu)
						(1+\beta\rho)}}}
			{1+\sqrt{\frac{(\alpha+\beta)(\rho+\mu)}{(1+\alpha\mu)
						(1+\beta\rho)}}}\le 
			\min\left\{\frac{1-\sqrt{\alpha\rho}}{1+\sqrt{\alpha\rho}},
			\frac{1-\sqrt{\beta\mu}}{1+\sqrt{\beta\mu}}\right\}
		\end{equation}
		and the inequality is strict because $\max\{\alpha\rho,\beta\mu\}<1$ and 
		$\min\{\rho+\mu,\alpha+\beta\}>0$.
		
		\item\label{rem:rates2} The optimal rate obtained in Theorem~\ref{t:main} is lower than the rate of convergence obtained for FISTA in \cite[Corollary~4.1]{Briceno2025Fista}. Indeed, since, $\frac{(\alpha+\beta)(\rho+\mu)}{(1+\alpha\mu)(1+\beta\rho)}\ge 
		\max\left\{\frac{\alpha(\rho+\mu)}{1+\alpha\mu},\frac{\beta(\rho+\mu)}{1+\beta\rho}\right\}$, we deduce
		\begin{align*}
			r^*\le 
			\min\left\{\frac{1-\sqrt{\frac{\alpha(\rho+\mu)}{1+\alpha\mu}}}{1+\sqrt{\frac{\alpha(\rho+\mu)}{1+\alpha\mu}}},
			\frac{1-\sqrt{\frac{\beta(\rho+\mu)}{1+\beta\rho}}}{1+\sqrt{\frac{\beta(\rho+\mu)}{1+\beta\rho}}}\right\}
			\leq
			\min\left\{{1-\sqrt{\frac{\alpha(\rho+\mu)}{1+\alpha\mu}}},
			{1-\sqrt{\frac{\beta(\rho+\mu)}{1+\beta\rho}}}\right\}
		\end{align*}
		and the inequality is strict because $\max\{\alpha\rho,\beta\mu\}<1$ and 
		$\min\{\rho+\mu,\alpha+\beta\}>0$.
		
		\item In the particular case when $\rho>0$, $\beta>0$, and 
		$\alpha=\mu=0$, the optimal rate obtained in Theorem~\ref{t:main}
		reduces to 
		\begin{equation}
			r^*=\frac{\sqrt{1+\beta\rho}-\sqrt{\beta\rho}}
			{\sqrt{1+\beta\rho}+\sqrt{\beta\rho}}
		\end{equation}
		which is strictly lower than the bound obtained in 
		\cite[Theorem~5.6]{Giselsson2017}, when $B=\partial g$ is 
		$\beta$-cocoercive and $A=\partial f$ is $\rho$-strongly 
		monotone, which is $\frac{1}{1+\sqrt{\beta\rho}}$. Hence, $r^*$ is 
		the best convergence rate available in 
		the literature for Peaceman--Rachford for minimizing 
		Problem~\ref{prob:main} when $f$ is $\rho$-strongly 
		convex and $\nabla g$ is $\beta^{-1}$-Lipschitz continuous.
		
		\item 
		A simple choice of parameters achieving the optimal convergence 
		rate is to consider 
		\begin{equation}
			\delta^*=\frac{\alpha\mu-\beta\rho}
			{\beta(1+\alpha\mu)+\alpha(1+\beta\rho)}\in\left]-\rho,\mu\right[,
		\end{equation}
		which leads $\eta^*=0$ in \eqref{e:eta}
		and, from \eqref{e:tauhat} (or \eqref{e:taustar})
		\begin{equation}
			\tau^*=\frac{\beta(1+\alpha\mu)+\alpha(1+\beta\rho)}
			{\sqrt{(\alpha+\beta)(\rho+\mu)(1+\alpha\mu)(1+\beta\rho)}}.
		\end{equation}
		In this case, the algorithm in \eqref{e:DRthep}  
		reduces to 
		\begin{equation}
			\label{e:DRthepred}
			(\forall n\in\N)\quad 
			\begin{cases}
				x_n=\prox_{\frac{\tau^*}{1+\delta^*\tau^*}f}
				\left(\frac{z_n}{1+\delta^*\tau^*}\right)\\
				y_n=2x_n- z_n\\
				p_n=\prox_{\frac{\tau^*}
					{1-\delta^*\tau^*}g}
				\left(\frac{y_n}{1-\delta^*\tau^*}
				\right)\\
				z_{n+1}=z_n+2(p_n-x_n).
			\end{cases}
		\end{equation}
	\end{enumerate}
\end{remark}

\section{Numerical experiments}
\label{sec:4}
In this section, we present numerical experiments to illustrate the advantages of our approach. First, we introduce an academic example and then we provide some applications to image processing.
\subsection{Academic Example}
Consider Problem~\ref{prob:main} in the case when $\HH=\RR^m$, 
\begin{equation}
	f\colon x\mapsto \frac{1}{2}\|Ax-a\|^2,\quad \text{and}\quad 
	g\colon x\mapsto \frac{1}{2}\|Bx-b\|^2,
\end{equation}
where $A$ and $B$ are non-zero $n\times m$ and $p\times m$ real 
matrices, respectively, $a\in\RR^n$, and $b\in\RR^p$.
Note that, since $\nabla f\colon x\mapsto A^{\top}(Ax-a)$ and $\nabla g\colon x\mapsto B^{\top}(Bx-b)$ are $\|A\|^2$- and $\|B\|^2$-Lipschitz continuous, 
then $\alpha=1/\|A\|^2>0$ and $\beta=1/\|B\|^2>0$ in view of Theorem~\ref{t:BH}. The strong convexity of $f$ (resp. $g$) with $\rho>0$ (resp. $\mu>0$) is guaranteed when $A$ (resp. $B$) is injective, which is only possible if $m\le n$ (resp. $m\le p$). In this setting, the strong convexity parameter of $f$ and $g$ corresponds to $\lambda_{\min}(A^{\top}A)$ and $\lambda_{\min}(B^{\top}B)$\footnote{Given a symmetric matrix $D$, $\lambda_{\min}(D)$ denotes the lowest real eigenvalue of $D$.}, respectively. 

In the numerical experiment below, we consider $10$ different dimension configurations for $(m,n,p)$ detailed in Table~\ref{tab:ae1} and, for each one, we generate $30$
random instances for matrices $A$ and $B$, generated in Julia by considering $A=0.5\cdot{\rm rand}(m,n)$
and $B=15\cdot{\rm rand}(m,p)$, where the function ${\rm rand}(m,n)$ generates a dense $m\times n$ matrix with independent random numbers uniformly distributed in $]0,1[$. We first compare with the average number of iterations and computational time to achieve $\|z_k-z^*\|\le tol=10^{-10}$ for the classical PRS when either $f$ or $g$ are strongly convex by using the optimal step-size derived in  
\cite[Proposition~3]{Giselsson2017IEEE}. We call PRS1 to PRS with optimal step-size when $f$ is strongly convex and PRS2 when $g$ is. Note that in the cases when $m>n$ (resp. $m>p$), PRS1 (resp. PRS2) cannot be implemented because $\rho=0$ (resp. $\mu=0$), in view of \cite[Proposition~3]{Giselsson2017IEEE}. In Table~\ref{tab:ae1} we observe an important reduction in the number of iterations and computational time except for the dimensions $(20,20,10)$ and $(40,40,20)$, in which the performances of PRS lev and PRS1 are very similar. In this case, note that the strong convexity of the problem is not relevant and the average values of the parameters lead to $r^*=0.990604$, which is very close to $1$ and the optimal rate in 
\cite[Proposition~3]{Giselsson2017IEEE} is $0.990615$. When the strong convexity is more important we can observe a reduction in the average of iterations up to $96.5\%$. In Figure~\ref{fig:cinco-subfiguras-iguales} we show a plot of the error versus iterations for a choice of particular instances in each dimension.

\begin{table}[h!]
	\centering
	\setlength{\tabcolsep}{2.5pt}
	\begin{tabular}{ccccccccccc}
		\multicolumn{5}{c}{ }& \multicolumn{2}{c}{\bf PRS lev} & \multicolumn{2}{c}{\bf PRS 1}& \multicolumn{2}{c}{\bf PRS 2}\\
		\textbf{$(m,n,p)$} & \textbf{avg $\rho$} & 
		\textbf{avg $\alpha$} & 
		\textbf{avg $\mu$} & \textbf{avg $\beta$} & Iter. & Time (ms) & Iter. & Time (ms) & Iter. & Time (ms)  \\
		\hline
		(20, 10, 20) & 0.0 & 0.077 & 0.883 & 4.2e-5 & 91.4 
		& 1.14  &  -  & - & 1396.9  &20.4  \\
		(20, 20, 10) & 5.7e-4 & 0.039 & 0.0 & 8.4e-5 & 
		1875.4  &25.7 & 1874.0  &25.7 & -  &- \\
		(20, 20, 20) & 5.7e-4 & 0.039 & 0.883 & 4.2e-5 
		& 130.5  & 1.85 & 5179.8  & 71.0 & 1029.0  & 13.6  \\
		(20, 40, 20) & 0.105 & 0.020 & 0.883 & 4.2e-5 & 
		107.2  & 1.48 & 252.8  & 3.43  &  688.0  & 9.5 \\
		(20, 20, 40) & 5.7e-4 & 0.039 & 87.3 & 2.2e-5 & 
		8.6  & 0.19  & 5317.5  & 71.8 & 237.1  & 3.5 \\\hline
		(40, 20, 40) & 0.0 & 0.020 & 0.264 & 1.1e-5 & 745.3 
		& 34.3 & -  & -  & 2093.3  & 78.07 \\
		(40, 40, 20) & 3.7e-4 & 9.8e-3 & 0.0 & 2.2e-5 & 
		3981.2  & 145.2 & 3978.6  & 144.7 & -  & -  \\
		(40, 40, 40) & 3.7e-4 & 9.8e-3 & 0.264 & 1.1e-5 
		& 893.2  & 34.3  & 7683.7  & 289.7  & 1119.6  & 42.1 \\
		(40, 80, 40) & 0.189 & 4.9e-3 & 0.264 & 1.1e-5 & 
		278.7  & 11.46 & 377.5  & 15.6  & 634.7  & 26.4 \\
		(40, 40, 80) & 3.7e-4 & 9.8e-3 & 153.9 & 5.5e-6 & 
		11.6  & 0.91 & 7856.5  & 515.6  & 340.1  & 19.2 \\\hline
	\end{tabular}
	\caption{Average parameters, number of iterations and 
		computational time (ms) for PRS lev, PRS 1, and  PRS 2 to reach 
		a tolerance of $\|z_k-z^*\|<10^{-10}$ in 30 random 
		instances of $A$ and $B$ with 
		%seed $1234+t$, , $niter=10000$
		$a,b=0$, $tol=10^{-10}$.}
	\label{tab:ae1}
\end{table}

\begin{figure}[ht]
	\centering
	% --- Primera fila: 2 imágenes ---
	\begin{subfigure}[t]{0.5\textwidth}
		\centering
		\includegraphics[width=\linewidth]{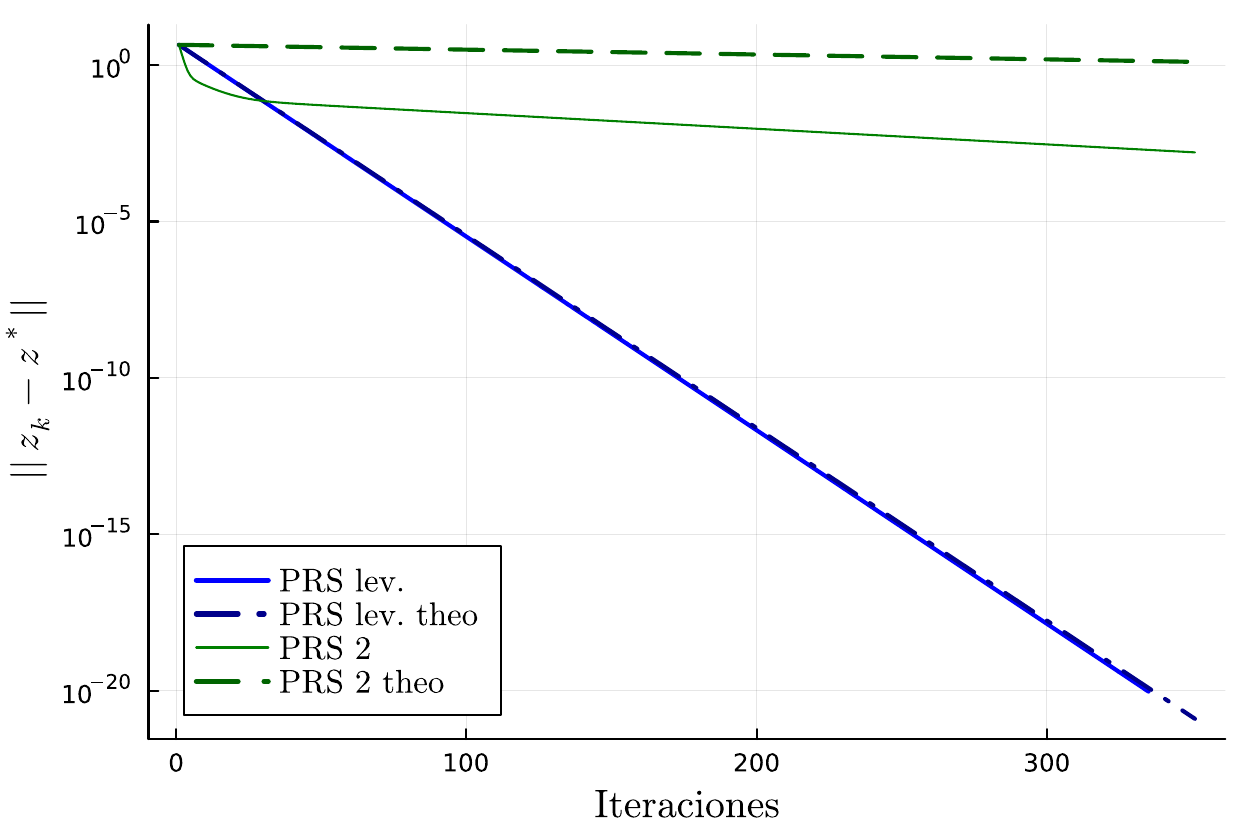}
		\caption{$(m,n,p)=(20,10,20)$}
		\label{fig:a}
	\end{subfigure}\hfill
	\begin{subfigure}[t]{0.5\textwidth}
		\centering
		\includegraphics[width=\linewidth]{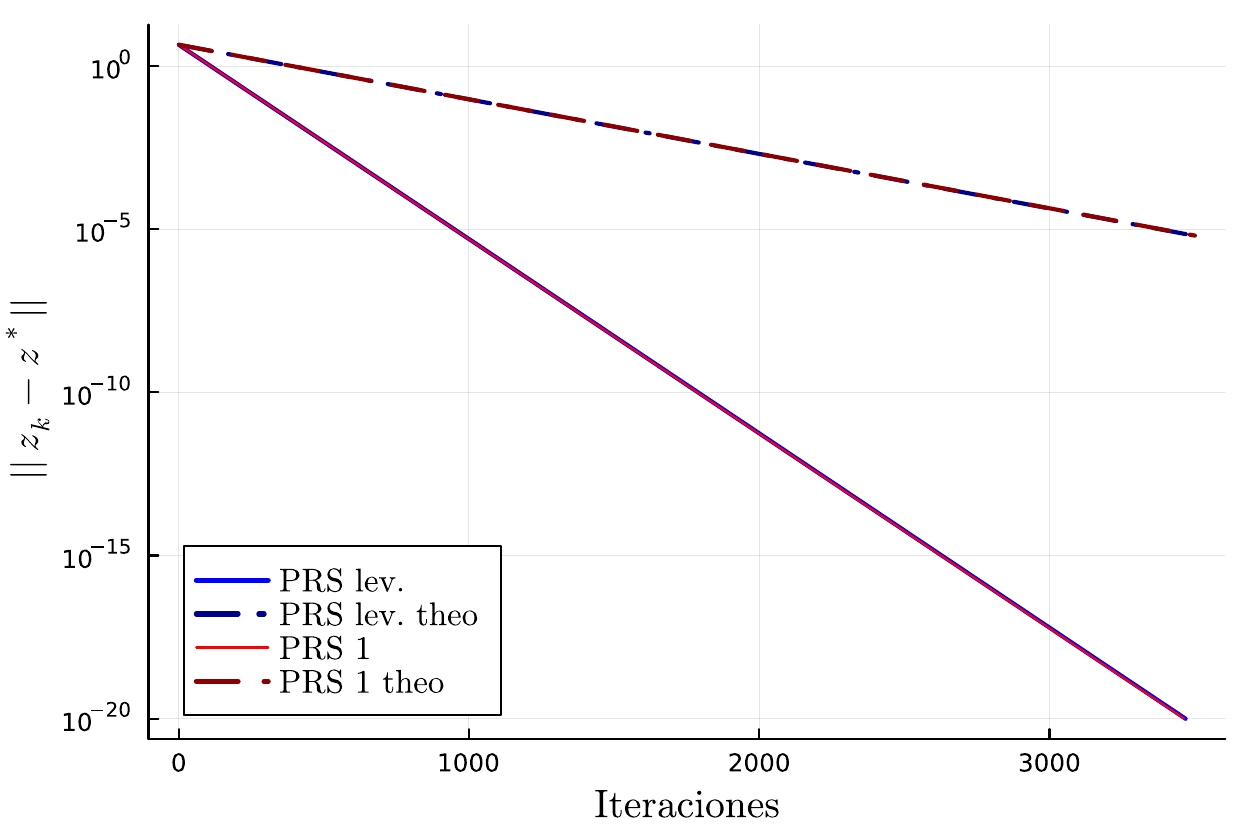}
		\caption{$(m,n,p)=(20,20,10)$}
		\label{fig:b}
	\end{subfigure}
	
	\vspace{0.8em}
	
	% --- Segunda fila: 1 imagen centrada ---
	\begin{subfigure}[t]{0.5\textwidth}
		\centering
		\includegraphics[width=\linewidth]{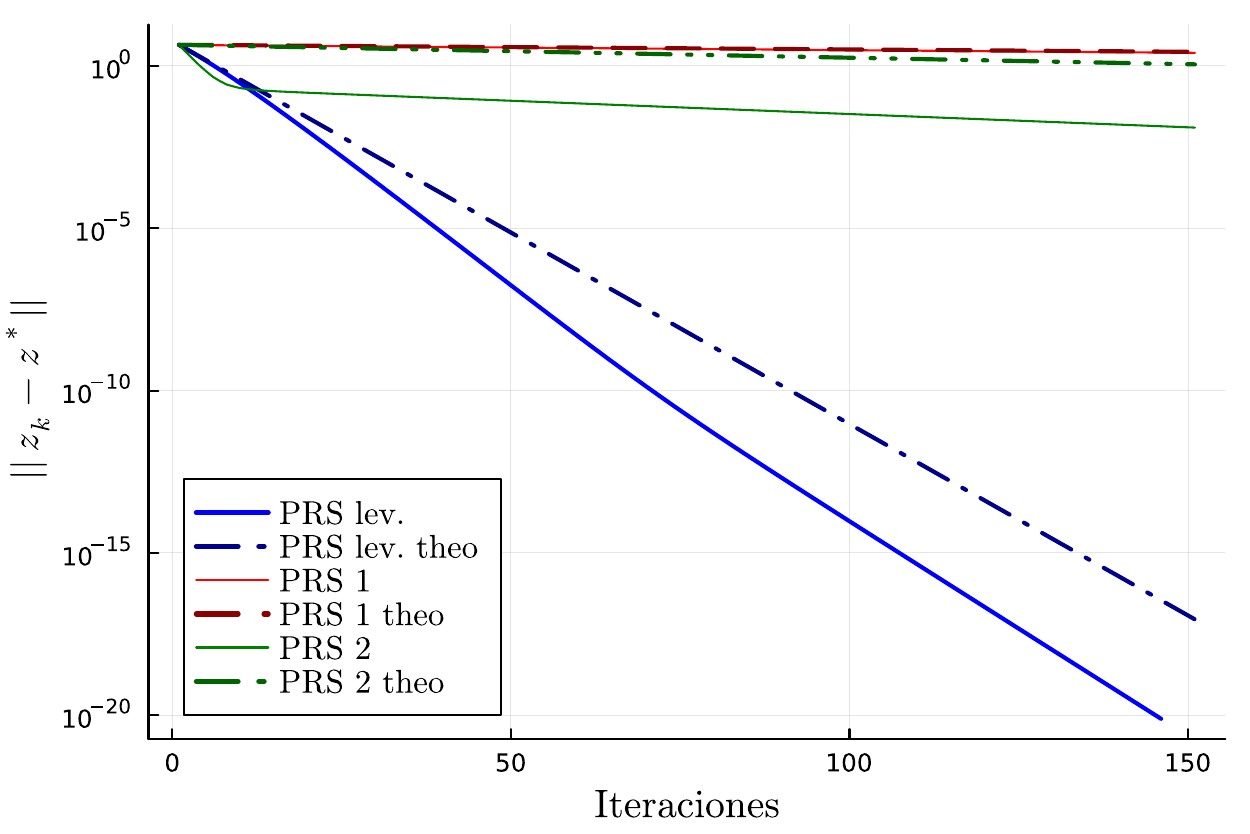}
		\caption{$(m,n,p)=(20,20,20)$}
		\label{fig:c}
	\end{subfigure}
	
	\vspace{0.8em}
	
	% --- Tercera fila: 2 imágenes ---
	\begin{subfigure}[t]{0.5\textwidth}
		\centering
		\includegraphics[width=\linewidth]{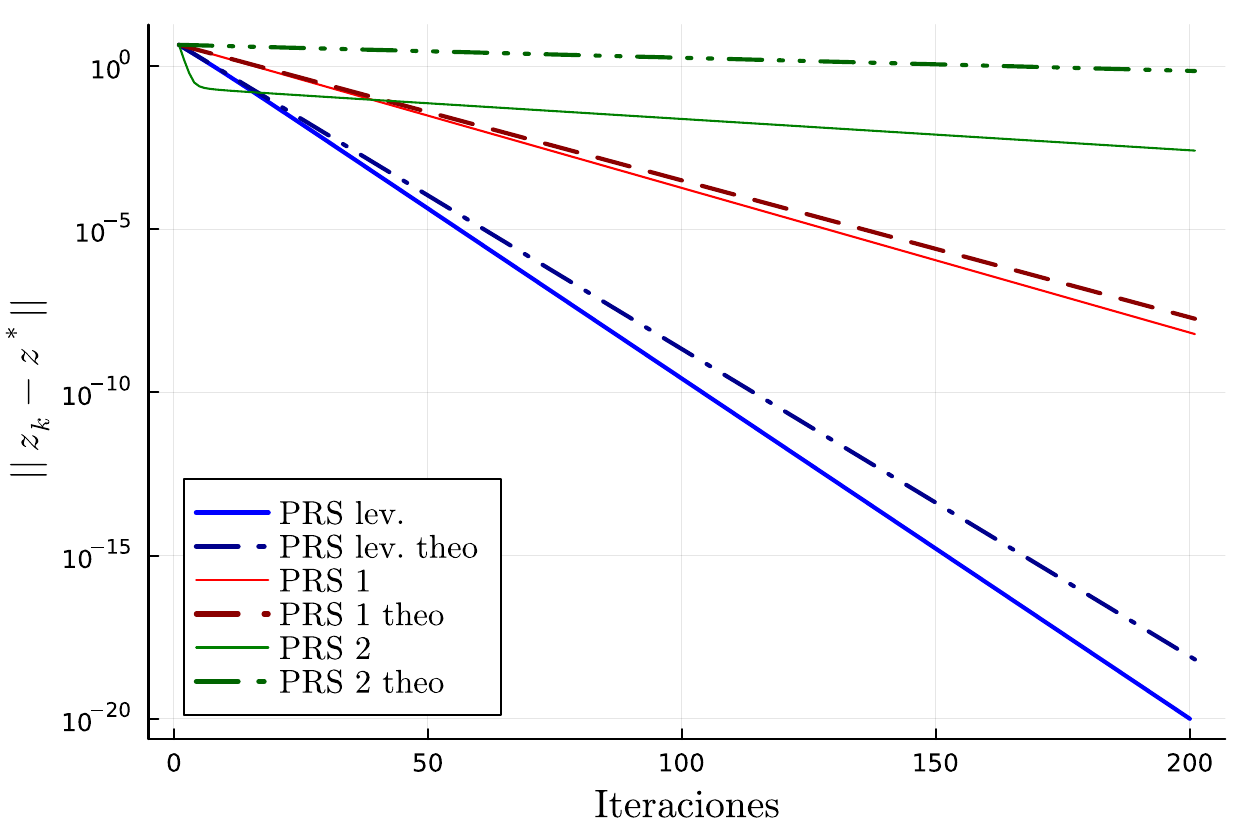}
		\caption{$(m,n,p)=(20,40,20)$}
		\label{fig:d}
	\end{subfigure}\hfill
	\begin{subfigure}[t]{0.5\textwidth}
		\centering
		\includegraphics[width=\linewidth]{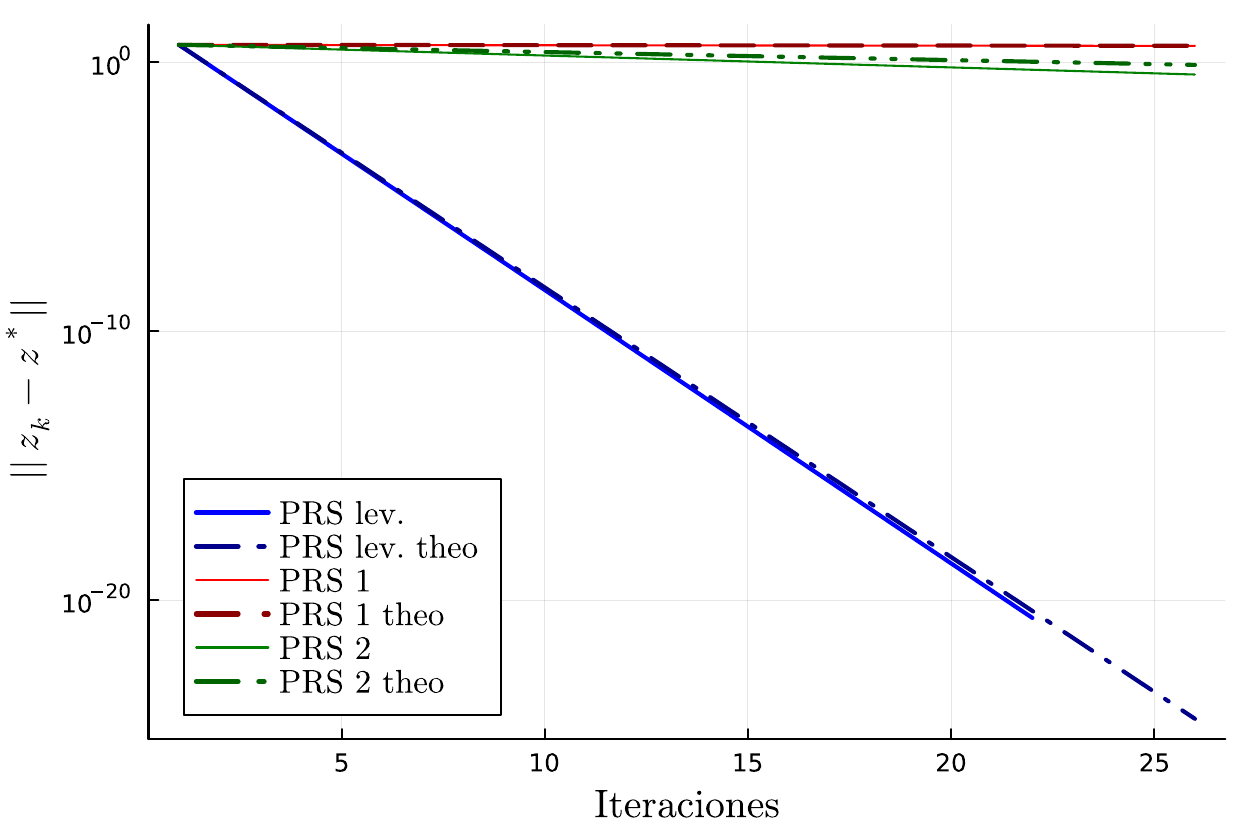}
		\caption{$(m,n,p)=(20,20,40)$}
		\label{fig:e}
	\end{subfigure}
	
	\caption{Theoretical bounds and numerical errors for random 
		instances in different dimensions.}
	\label{fig:cinco-subfiguras-iguales}
\end{figure}

\subsection{Image processing}
Let $(n_1,n_2,m_1,m_2) \in \N^4$ and let $x \in \R^{n_1\times n_2}$ be an image to be recovered from 
an observation
\begin{equation}\label{eq:modelim}
	b = Tx+\epsilon,
\end{equation}	
where $T\colon \R^{n_1\times n_2}\to \R^{m_1\times m_2}$ is an operator representing the observation process and $\epsilon$ 
is an additive noise. The original image $x$ can be approximated by solving Problem~\ref{prob:main}  
by setting $\HH = \R^{n\times n} $,
\begin{equation}
	f\colon x\mapsto \frac{1}{2}\|Tx-b\|^2,\quad \text{and}\quad 
	g\colon x\mapsto\lambda H^\varepsilon(Wx),
\end{equation}
where $\lambda \in \RPP$ is a regularization parameter, $H^{\varepsilon}:\R^{n_1\times n_2}\to \R$ is the Huber function defined by 
\begin{equation}\label{eq:def_huber}
	\left(\forall w = (\omega_{i,j})_{\substack{1\leq i \leq n_1\\1\leq j \leq n_2}} \in \R^{n_1\times n_2}\right) \quad H^{\varepsilon}(w) 
	=  \sum_{i=1}^{n_1}\sum_{j=1}^{n_2} h^{\varepsilon}(\omega_{i,j}), \quad  
\end{equation}
and
\begin{equation}
	(\forall \xi \in \R) \quad
	h^{\varepsilon} (\xi) = \begin{cases}
		|\xi|-\frac{\varepsilon}{2}, & \text{ if } |\xi | > \varepsilon,\\
		\frac{\xi^2}{2\varepsilon}, & \text{ otherwise}.
	\end{cases}
\end{equation}
and $W:\R^{n_1\times n_2}\to \R^{n_1\times n_2}$ is an orthogonal basis wavelet transform. Note that $f$ is $\lambda_{\min}(T^\top T)$-strongly convex, $\nabla f$ is $\|T\|^2$-Lipschitz continuous, $g$ is convex and $\nabla g= \lambda W^{\top} \nabla H^\varepsilon W$ is $\lambda/\varepsilon$-Lipschitz continuous. Thus, in the context of  Problem~\ref{prob:main}, we have $\rho = \lambda_{\min}(T^\top T)$, $\alpha = \|T\|^{-2}$, $\mu =0$, and $\beta=\varepsilon/\lambda$. We solve this problem using our approach, PRS with the optimal step-size in \cite[Proposition~3]{Giselsson2017IEEE}, and 
FISTA in its leveraged form for strongly convex problems \cite{BeckTeboulle2009,Briceno2025Fista}. 
Note that we can apply FISTA either with a backward step in $g$ and a forward step in $f$, or vice versa, which are called by FISTA 1 and FISTA 2, respectively.

We consider the scenarios where $T$ is a blur operator in image restoration and where $T$ is the discrete Radon transform in computed tomography (CT) \cite{Kak2001}. In each experiment, we use a stopping criterion consisting of a maximum of $10^3$ iterations and a {\it normalized error} with a tolerance of $10^{-12}$, i.e., given a solution $z^*$, the algorithm stops if
\begin{equation*}
	\frac{\|z_{n+1}-z^*\|}{\|z_0-z^*\|}<  10^{-12}.
\end{equation*}

\subsubsection{Image Restoration}
In this case, $T$ is a Gaussian point spread function of size $5\times 5$ and standard deviation $\sigma \in \RPP$ and $\epsilon$ is a zero mean Gaussian noise with variance $0.008$. We set $\varepsilon = 0.01$ and $\lambda = 0.07$. The operator $W$ is a level-3 Haar basis wavelet transform. As a test image, we use the $512\times 512$ image shown in Figure~\ref{fig:x_peppers}. To test different values of $\rho$, we consider $\sigma \in \{0.5,0.6\}$ which yield the values of $\rho$ described in Table~\ref{Tab:denoising}. This table also presents the results of our experiments. Note that PRS lev achieves better performance in terms of iteration number and CPU time in all cases when compared with PRS and the two versions of FISTA. Although its advantage over PRS is slight, Figure~\ref{fig:error_blur} illustrates both the theoretical and experimental benefits of the improved convergence rate of PRS lev in this case. The blurred and noisy image, as well as the denoised results for each case, are shown in Figures \ref{fig:s05} and \ref{fig:s06}, respectively.
\begin{table}[h!]
	\centering
	\setlength{\tabcolsep}{2.5pt}
	\begin{tabular}{cccccccccccccccc}
		\multicolumn{5}{c}{ }& \multicolumn{2}{c}{\bf PRS lev} & \multicolumn{2}{c}{\bf PRS} & \multicolumn{2}{c}{\bf FISTA 1} &\multicolumn{2}{c}{\bf FISTA 2} \\ \cline{6-13}
		$\sigma$ & $\rho$ & $\alpha$ & $\mu $& $\beta$  & Iter. & Time (s) & Iter. & Time (s) & Iter. & Time (s) & Iter. & Time (s)  \\\hline
		$0.5$ & 0.11 & 1 & 0 & 0.1429 & 39 & 1.38 & 42 & 1.54& 59 & 2.06 & 189 & 6.93 \\
		$0.6$ & 0.013 & 1 & 0 & 0.1429 & 114 & 4.04 & 122 & 4.43 & 167 & 5.77 & 517 & 18.36   \\
	\end{tabular}
	\caption{Results in terms of number of iterations and CPU time in second. We consider $tol=10^{-12}$, $niter=1000$}\label{Tab:denoising}
\end{table}

\begin{figure}
	\centering
	\subfloat[\scriptsize $\sigma = 0.5$]{\label{fig:resultado_05}\includegraphics[width=0.9\textwidth]{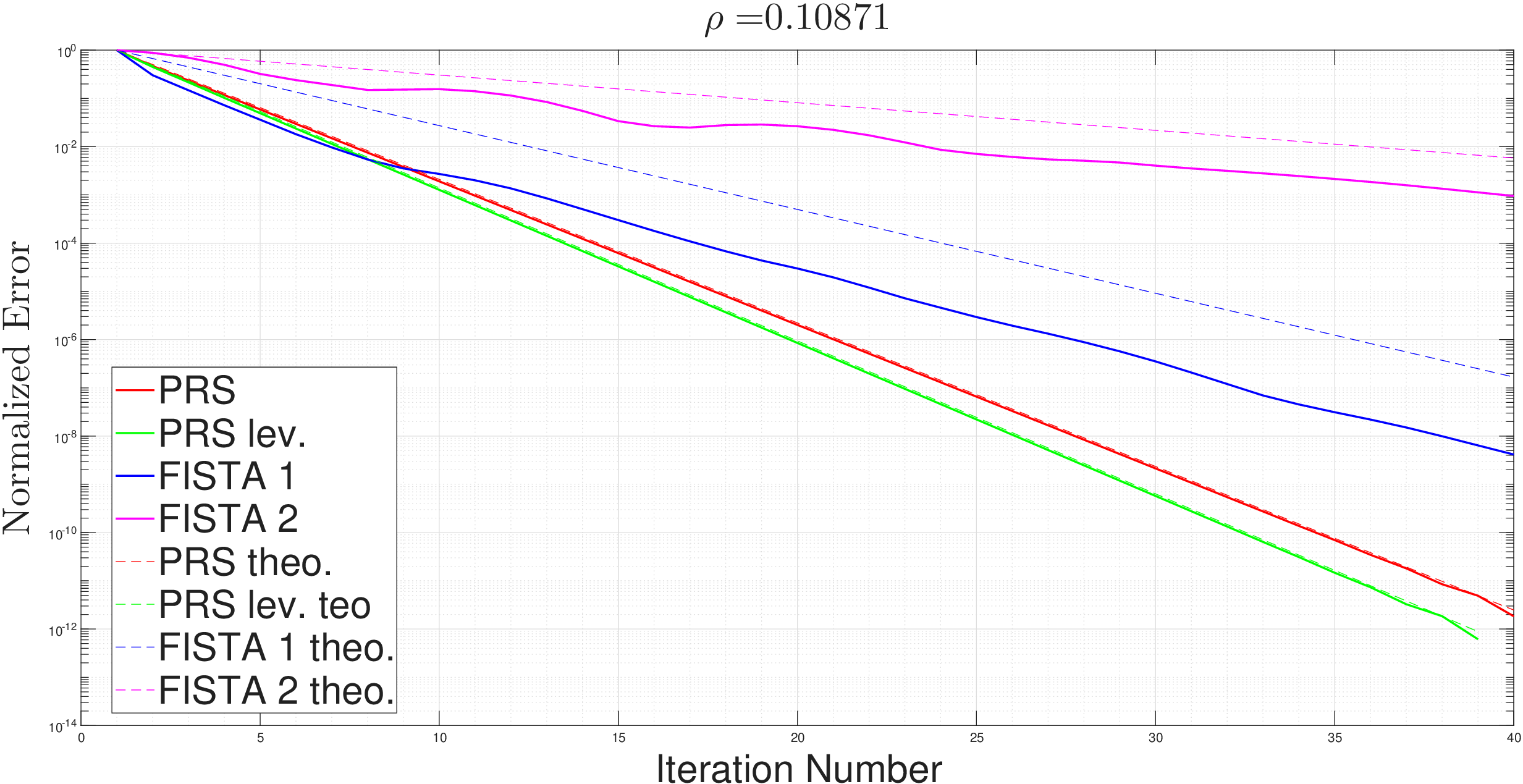}}\\[1cm]
	\subfloat[\scriptsize $\sigma = 0.6$]{\label{fig:resultado_06}\includegraphics[width=0.9\textwidth]{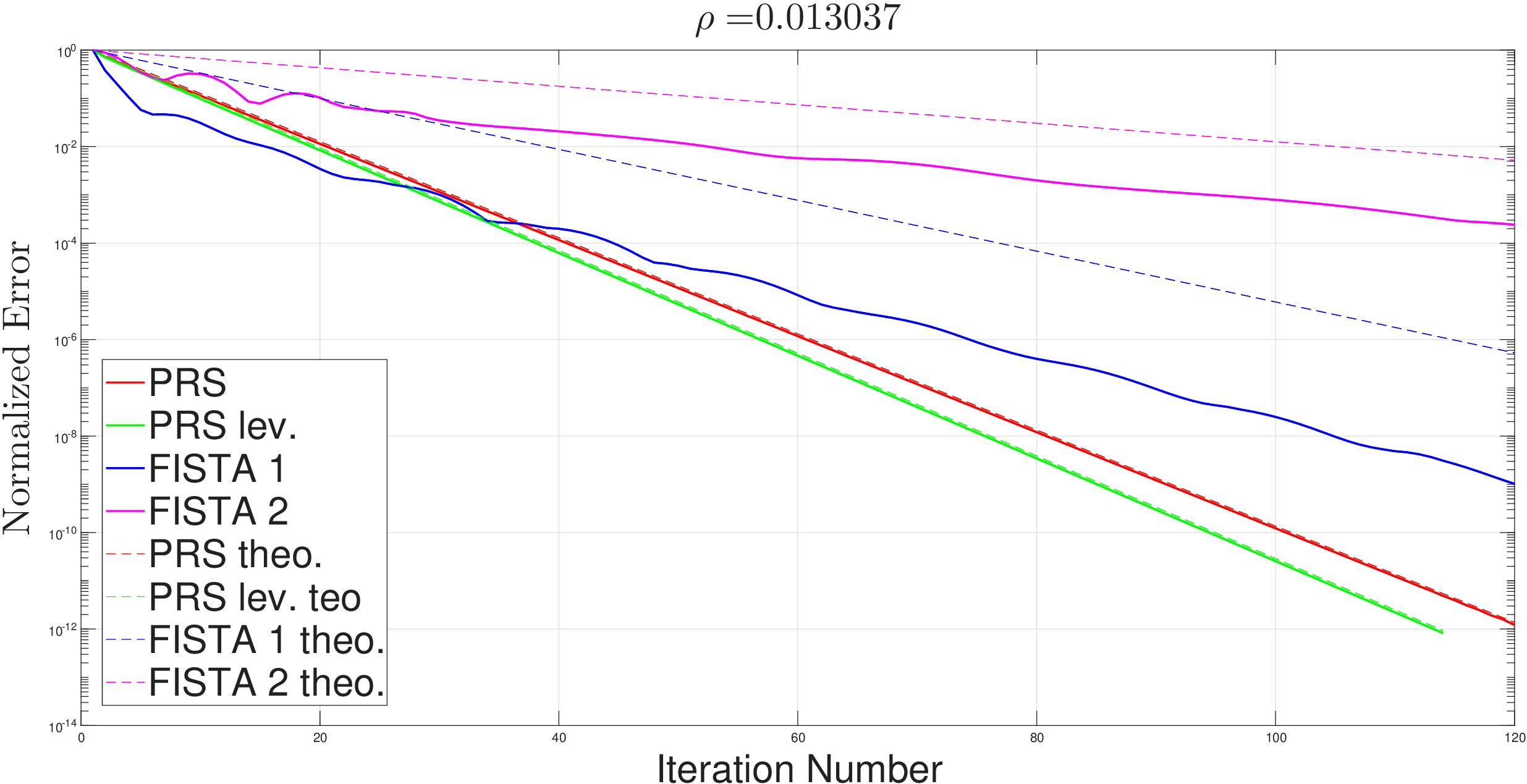}}\,
	\captionsetup{width=\textwidth} \caption{Normalized error vs Iteration number.} 
	\label{fig:error_blur}
\end{figure}
\begin{figure}
	\centering
	\subfloat[Deblurring]{\label{fig:x_peppers}\includegraphics[width=0.19\textwidth]{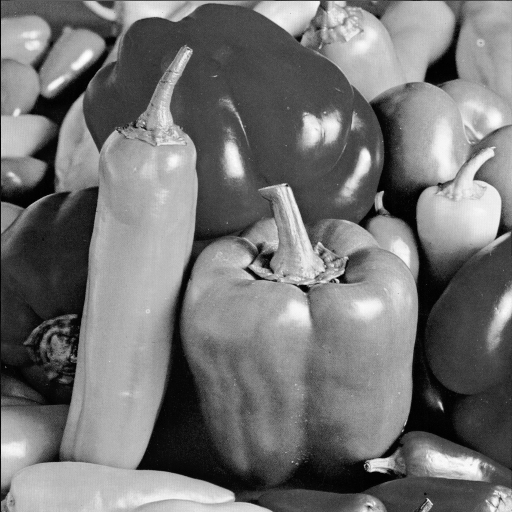}}\quad  \subfloat[C. T.]{\label{fig:x_phantom}\includegraphics[width=0.19\textwidth]{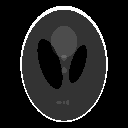}} 
	\label{fig:original} \caption{Original Images} 
\end{figure}

\begin{figure}
	\centering
	\subfloat[\scriptsize Blur and Noisy]{\label{fig:b_05}\includegraphics[width=0.19\textwidth]{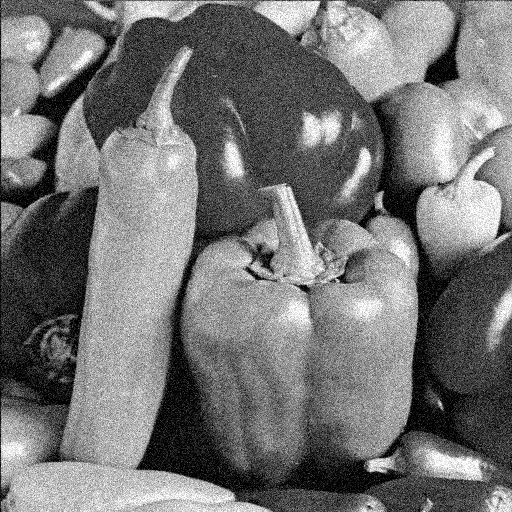}}\,
	\subfloat[\scriptsize PRS lev.]{\label{fig:x_PR_05}\includegraphics[width=0.19\textwidth]{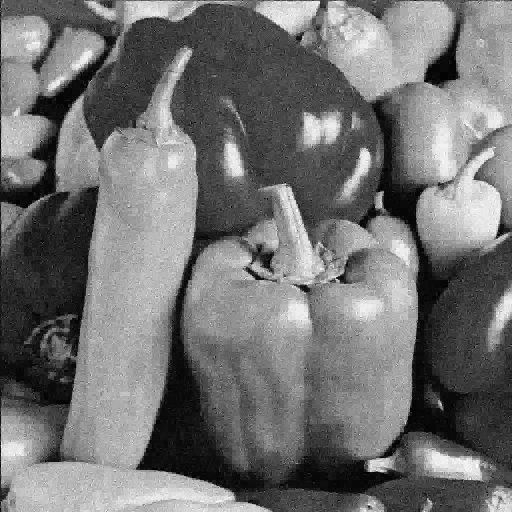}}\,
	\subfloat[\scriptsize PRS]{\label{fig:x_PRO_05}\includegraphics[width=0.19\textwidth]{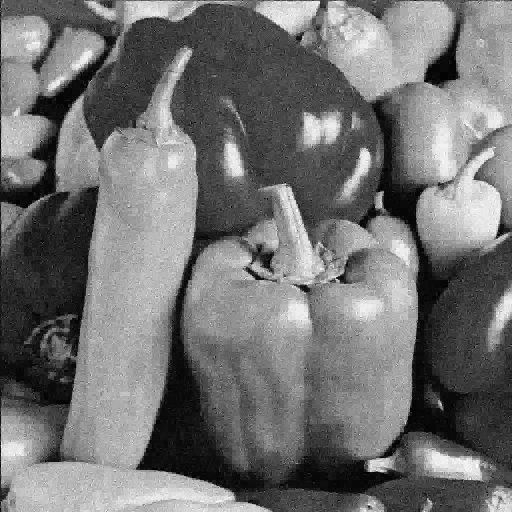}}\,
	\subfloat[\scriptsize FISTA 1]{\label{fig:x_FB_05}\includegraphics[width=0.19\textwidth]{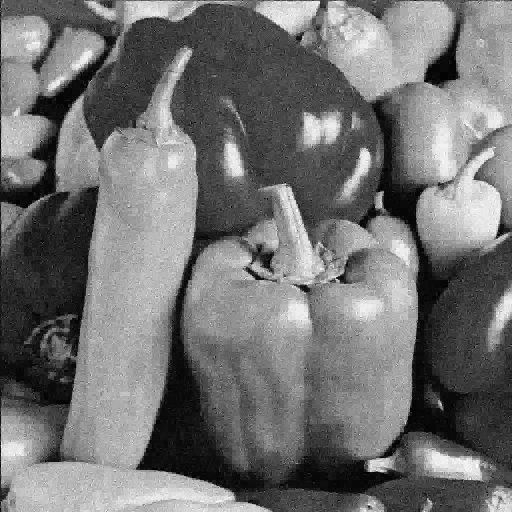}}\,
	\subfloat[\scriptsize FISTA 2]{\label{fig:x_FB2_05}\includegraphics[width=0.19\textwidth]{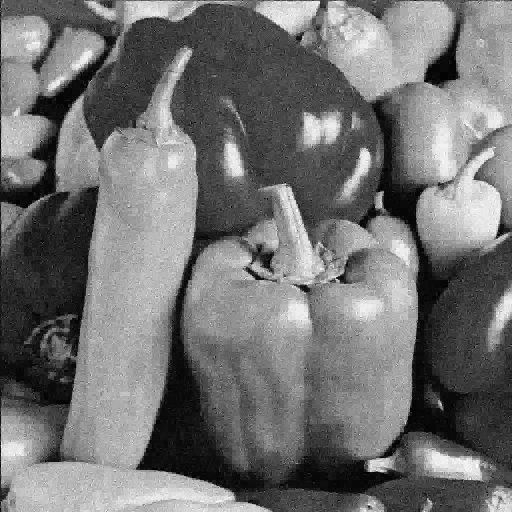}}
	\captionsetup{width=\textwidth} \caption{$\sigma =0.5$, $\rho = 0.11$} 
	\label{fig:s05}
\end{figure}

\begin{figure}
	\centering
	\subfloat[\scriptsize Blur and Noisy]{\label{fig:b_06}\includegraphics[width=0.19\textwidth]{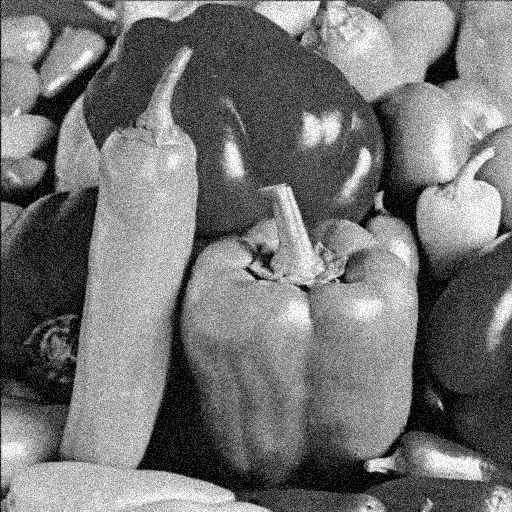}}\,
	\subfloat[\scriptsize PRS lev.]{\label{fig:x_PR_06}\includegraphics[width=0.19\textwidth]{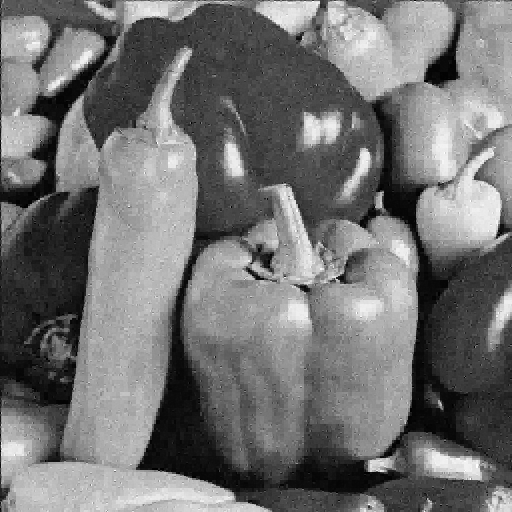}}\,
	\subfloat[\scriptsize PRS]{\label{fig:x_PRO_06}\includegraphics[width=0.19\textwidth]{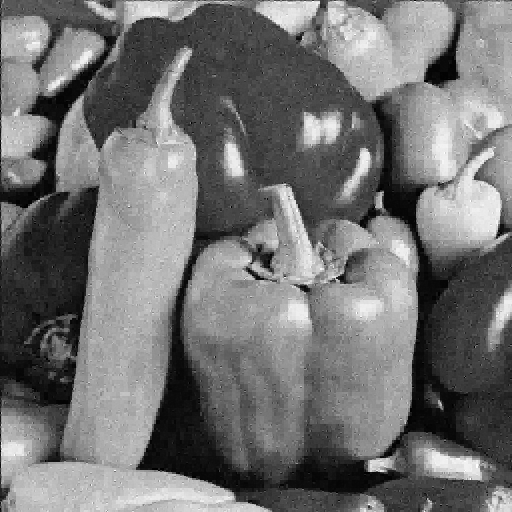}}\,
	\subfloat[\scriptsize FISTA 1]{\label{fig:x_FB_06}\includegraphics[width=0.19\textwidth]{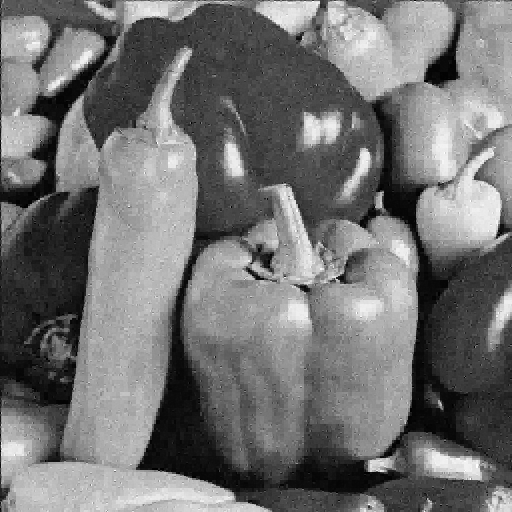}}\,
	\subfloat[\scriptsize FISTA 2]{\label{fig:x_FB2_06}\includegraphics[width=0.19\textwidth]{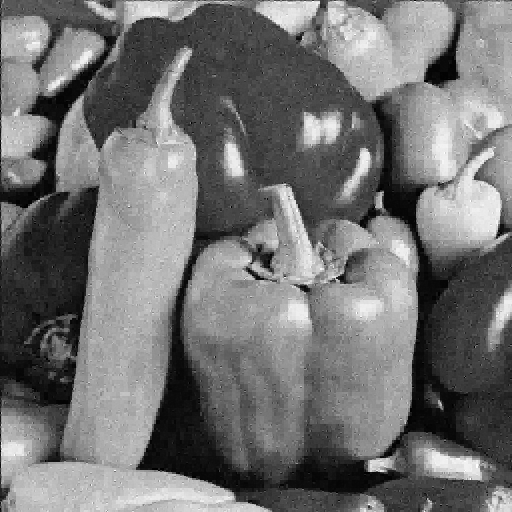}}
	\captionsetup{width=\textwidth} \caption{$\sigma =0.6$, $\rho = 0.013$} 
	\label{fig:s06}
\end{figure}

\subsubsection{Computed Tomography}
Here $T$ is the discretized Radon projector given by the line length ray-driven projector \cite{zengRayDriven93}. It was implemented in MATLAB using the line fan-beam projector provided by the
ASTRA toolbox \cite{vanAarleASTRA2016,vanAarleASTRA2015} with a fan-beam geometry over $180^\circ$, the source to object distance is 800 mm, and the source-to-image distance is 1200 mm. The test image $x$ is created by the \emph{phantom} function in MATLAB and we considered $n_1=n_2 =n\in \{64,128\}$. When $n = 64$, we consider $150$ projection angles and a detector with $249$ bins, for $n=128$, we consider $360$ projection angles and a detector with 249 bins, which results in sinograms of size $(m_1,m_2)=(112,150)$ and $(m_1,m_2)=(249,360)$, respectively. Moreover, we consider $\varepsilon = \lambda = 10^{-5}$ and $W$ as an orthonormal Symmlet basis wavelet transform of level $2$. The values of $\rho$, $\alpha$, and $\beta$, for each case, are described in Table~\ref{Tab:CT}. To implement PRS lev, PRS, and FISTA 2, we calculate $\prox_{f}$ using the MATLAB function {\it inv} once outside the iterative scheme (the time required for this computation is included in the total running time of each algorithm). From Table~\ref{Tab:CT} we observe that, in each case, PRS lev. exhibits the best overall performance. Note that PRS and FISTA reach the maximum number of iterations. Although FISTA 2 is competitive with PRS lev., Figure~\ref{fig:error_CT} shows that it is initially fast but subsequently follows its theoretical convergence rate, which is slower than that of PRS lev. The blurred and noisy image, along with the denoised reconstructions for each case, are displayed in Figures~\ref{fig:n64} and \ref{fig:n128}, respectively. We also observe that the reconstructions produced by PRS and FISTA 1 remain far from the target solution.
\begin{table}[h!]
	\centering
	\setlength{\tabcolsep}{2.5pt}
	\begin{tabular}{ccccccccccccc}
		\multicolumn{5}{c}{ }& \multicolumn{2}{c}{\bf PRS lev} & \multicolumn{2}{c}{\bf PRS} & \multicolumn{2}{c}{\bf FISTA 1} &\multicolumn{2}{c}{\bf FISTA 2} \\ \cline{6-13}
		$n$ & $\rho$ & $\alpha$ & $\mu $& $\beta$  & Iter. & Time (s) & Iter. & Time (s) & Iter. & Time (s) & Iter. & Time (s)  \\ \hline
		$64$ & 0.0128& $5.6 \cdot 10^{-5}$ & 0 & 1 & 124 & 1.31 & 1000 & 5.24 & 1000 & 12.47& 145 & 1.44  \\
		$128$ & 0.0026 & $1.13 \cdot 10^{-5}$ & 0 & 1 & 274 & 27.0 & 1000 & 51.51& 1000 & 189.0 & 334 & 31.0 \\
	\end{tabular}
	\caption{Results in terms of number of iterations and CPU time in second. We consider $tol=10^{-12}$, $niter=1000$.}\label{Tab:CT}
\end{table}

\begin{figure}
	\centering
	\subfloat[$n =64$]{\label{fig:resultado_n64}\includegraphics[width=0.9\textwidth]{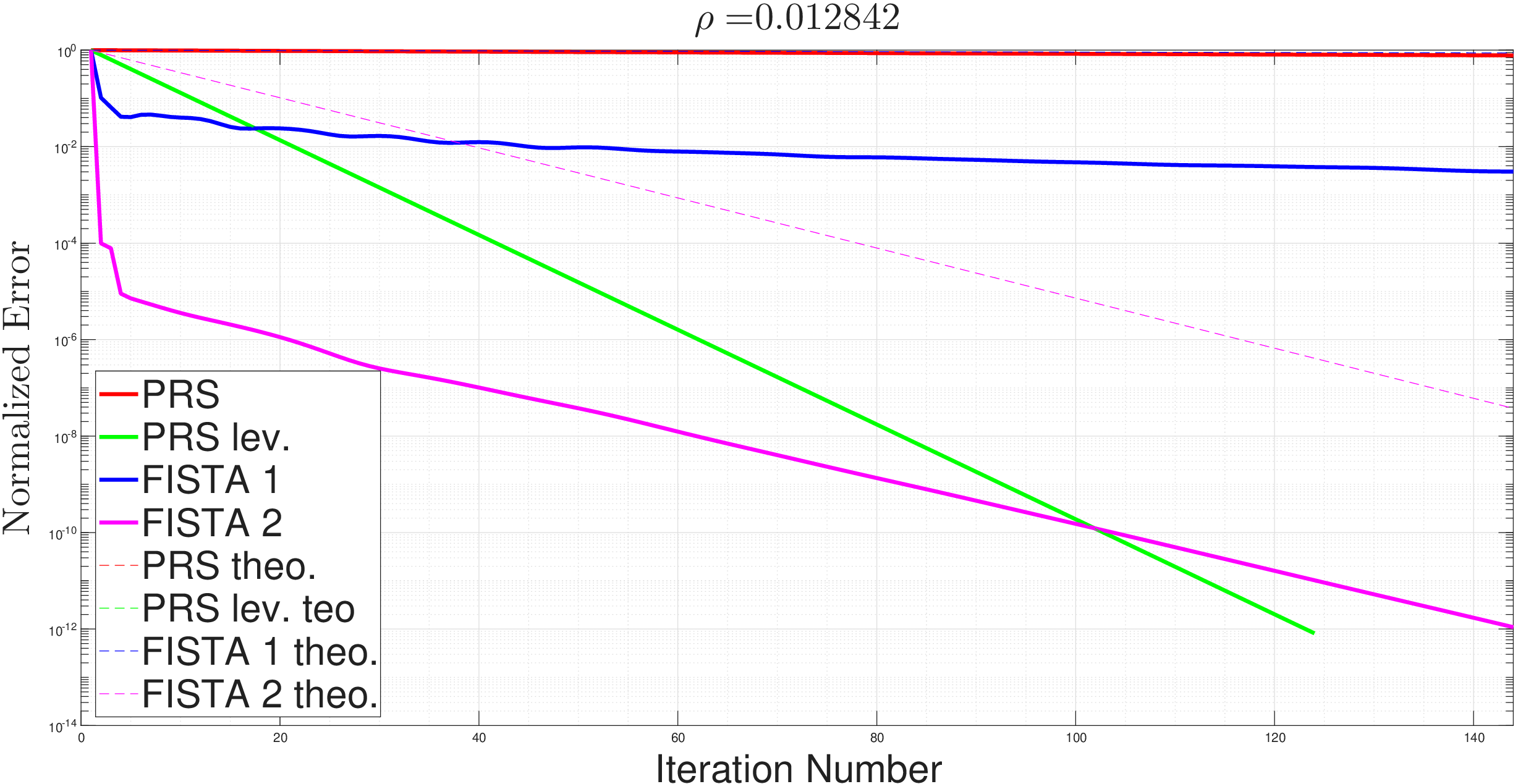}}\\	
	\subfloat[$n =128$]{\label{fig:resultado_n128}\includegraphics[width=0.9\textwidth]{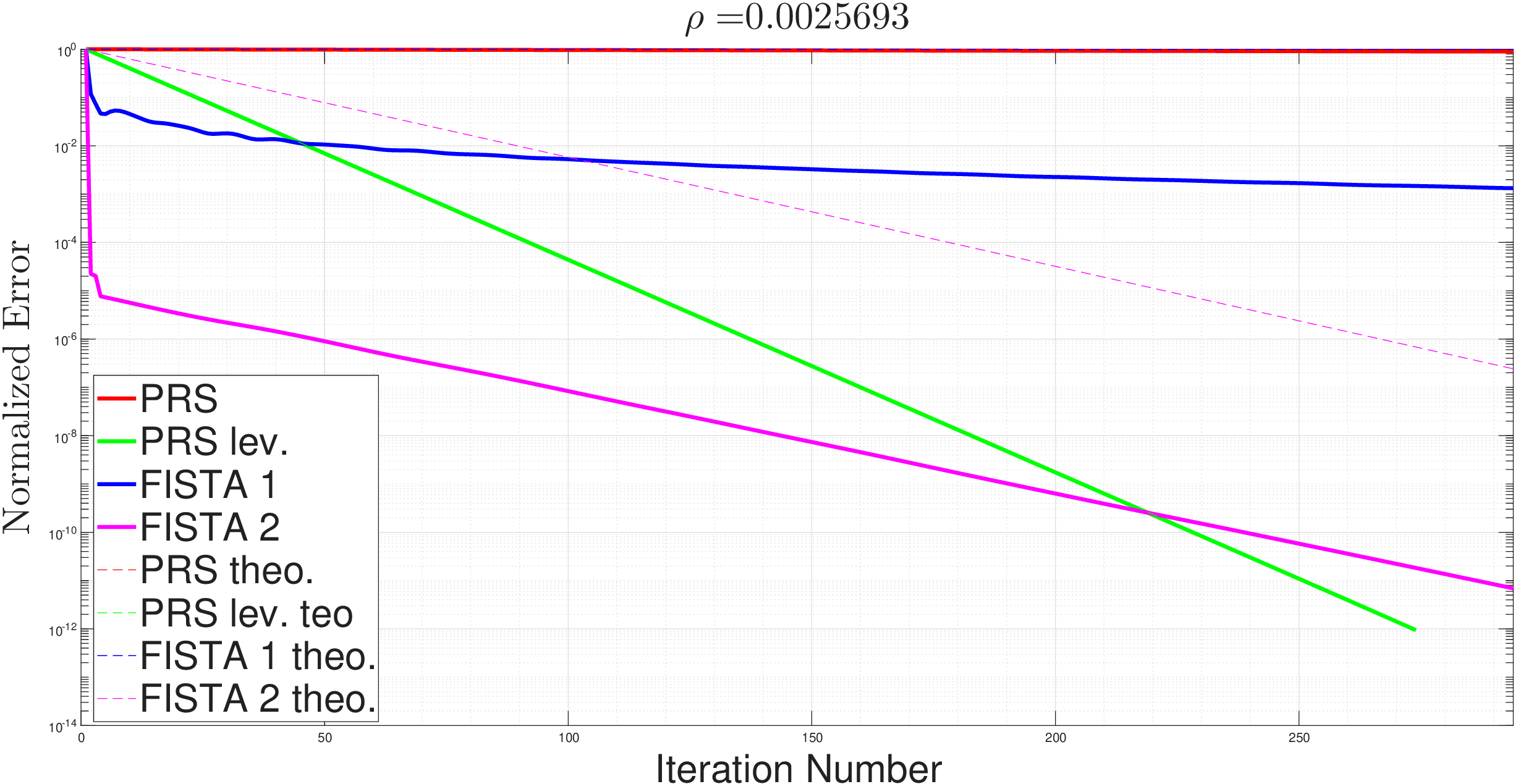}}\,
	\captionsetup{width=\textwidth} \caption{Normalized error vs Iteration number.} 
	\label{fig:error_CT}
\end{figure}

% \begin{figure}
	% 	\centering
	% 	\subfloat {\label{fig:x_phantom}\includegraphics[width=0.19\textwidth]{x_real_128.png}} \caption{Original Image}
	%     \label{fig:original_CT}
	% \end{figure}

\begin{figure}
	\centering
	\subfloat[\scriptsize Noisy Observation]{\label{fig:b_64}\includegraphics[width=0.19\textwidth]{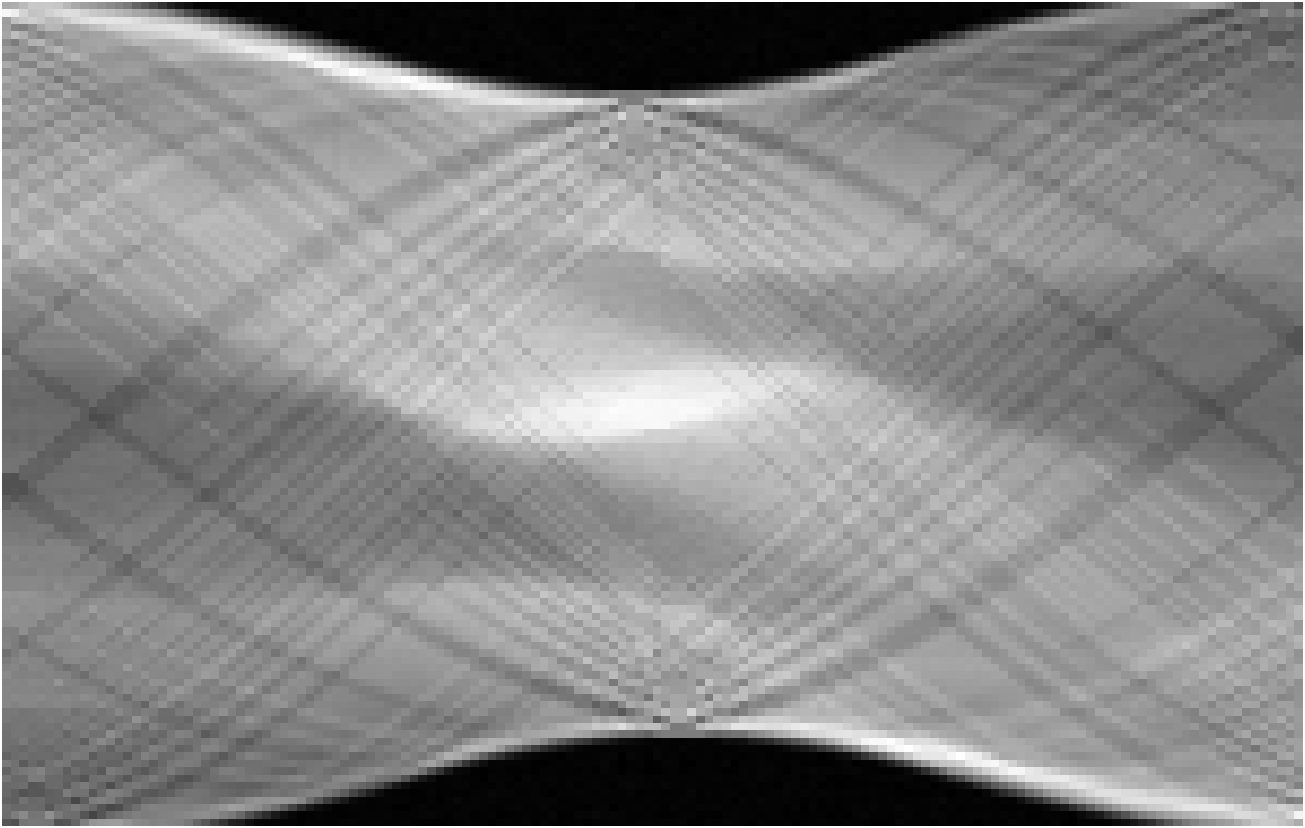}}\,
	\subfloat[\scriptsize PRS lev.]{\label{fig:x_PRO_64}\includegraphics[width=0.19\textwidth]{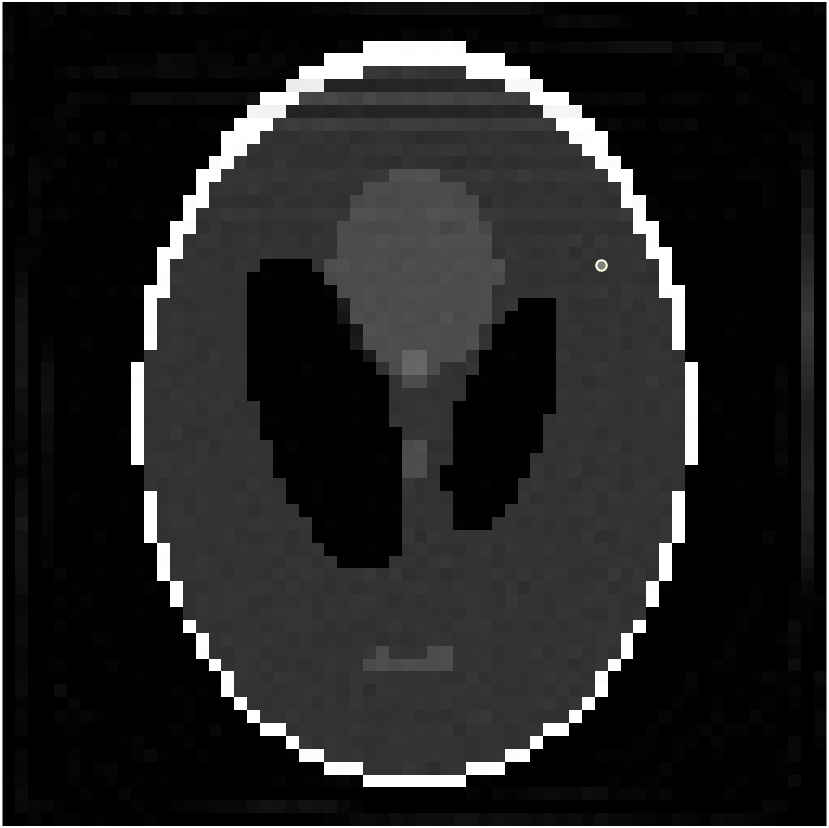}}\,
	\subfloat[\scriptsize PRS]{\label{fig:x_PR_64}\includegraphics[width=0.19\textwidth]{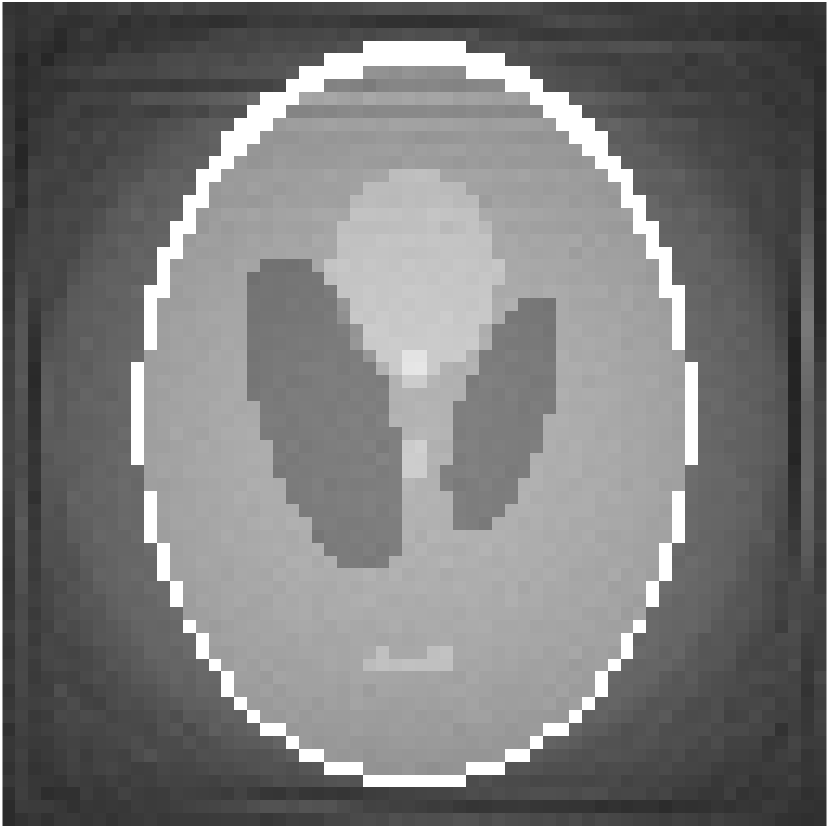}}\,
	\subfloat[\scriptsize FISTA 1]{\label{fig:x_FB_64}\includegraphics[width=0.19\textwidth]{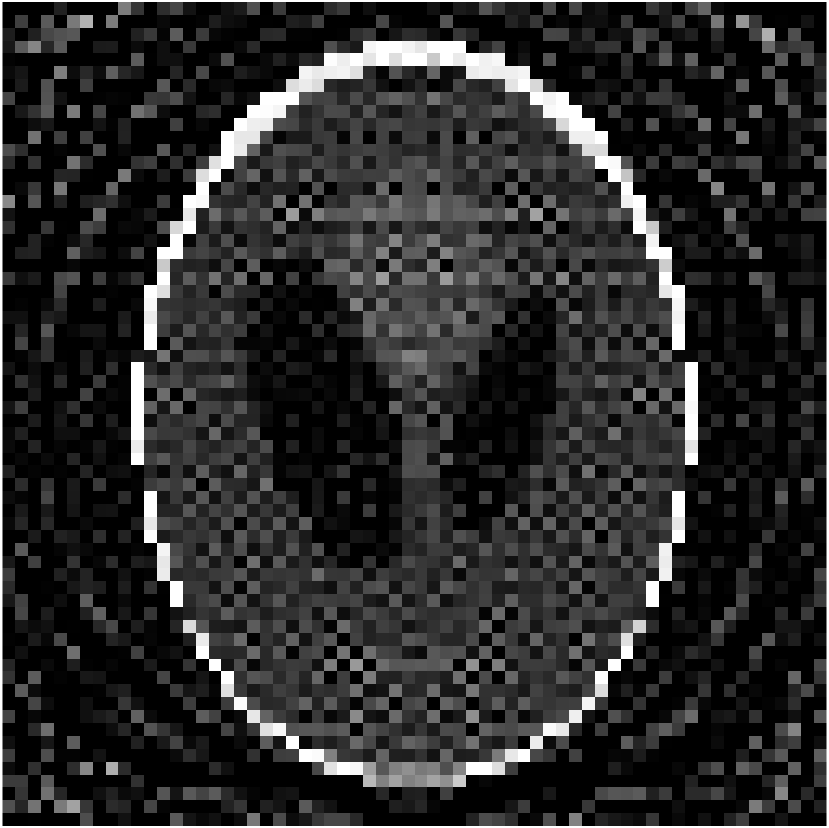}}\,
	\subfloat[\scriptsize FISTA 2]{\label{fig:x_FB2_64}\includegraphics[width=0.19\textwidth]{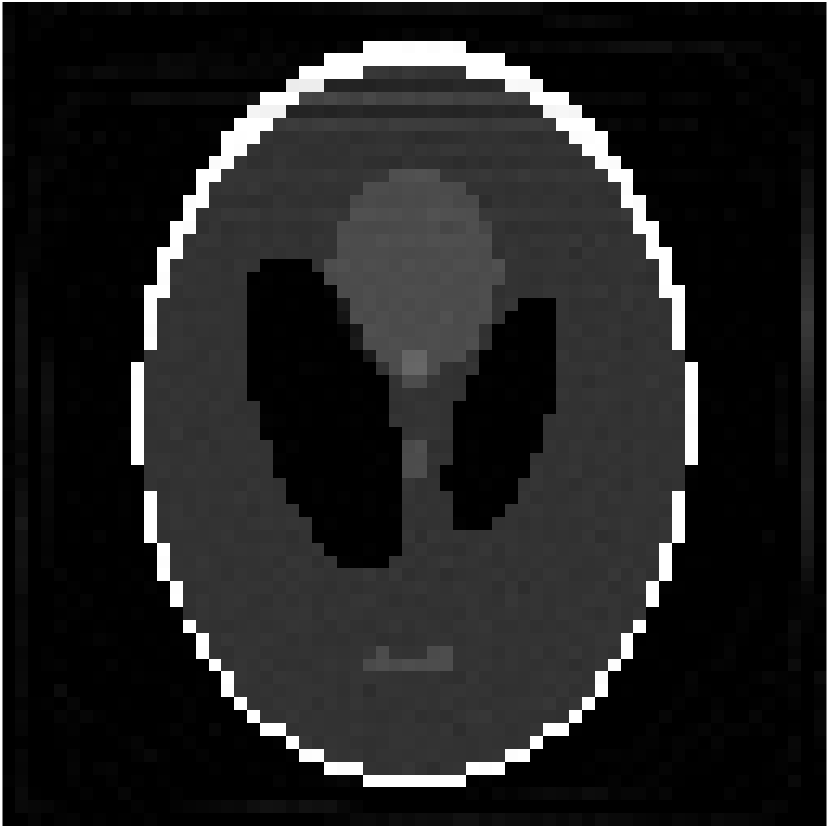}}
	\captionsetup{width=\textwidth} \caption{$m = 64$.} 
	\label{fig:n64}
\end{figure}

\begin{figure}
	\centering
	\subfloat[\scriptsize Noisy Observation]{\label{fig:b_128}\includegraphics[width=0.19\textwidth]{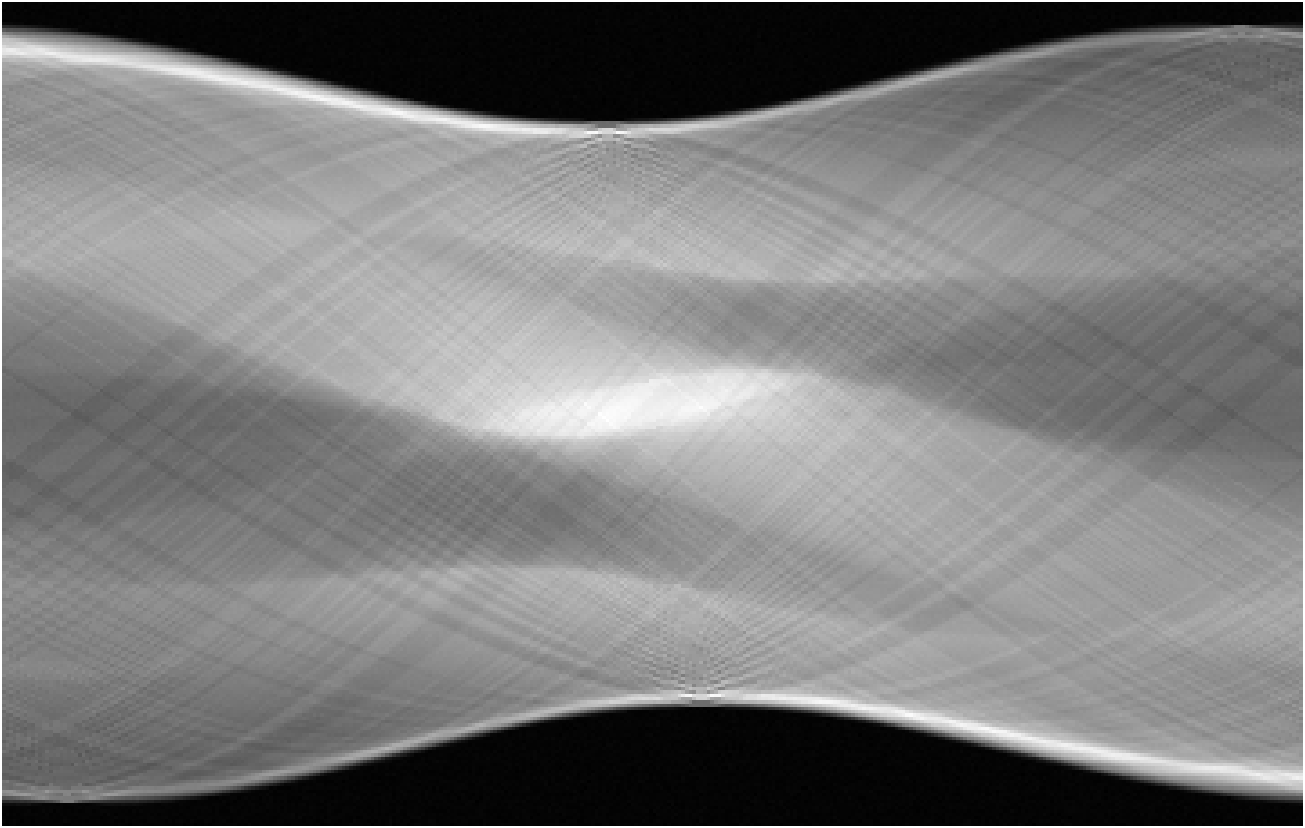}}\,
	\subfloat[\scriptsize PRS lev.]{\label{fig:x_PRO_128}\includegraphics[width=0.19\textwidth]{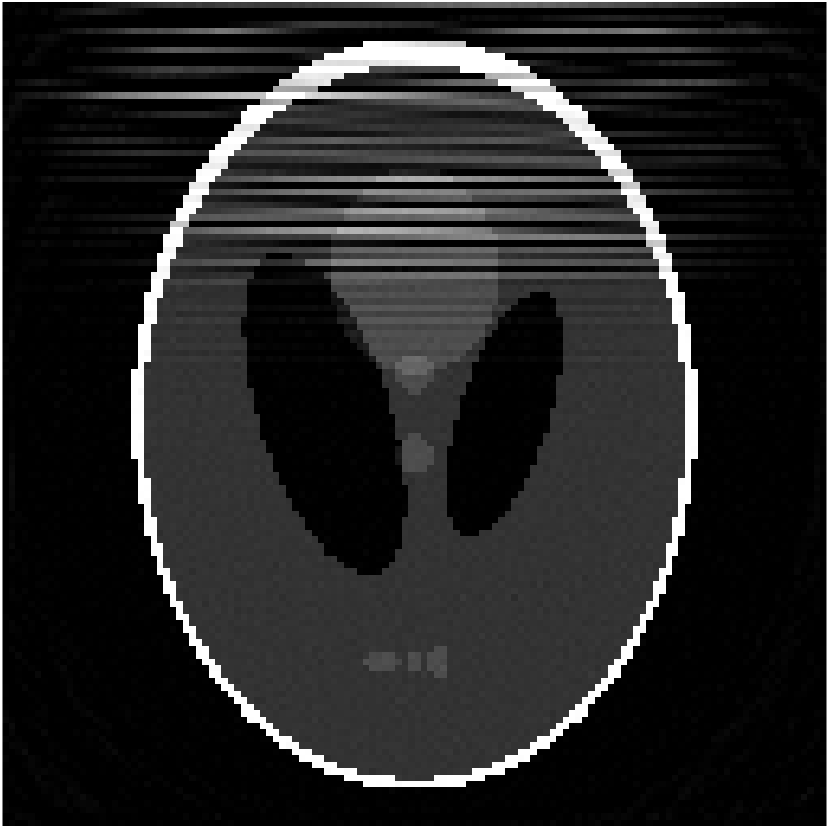}}\,
	\subfloat[\scriptsize PRS]{\label{fig:x_PR_128}\includegraphics[width=0.19\textwidth]{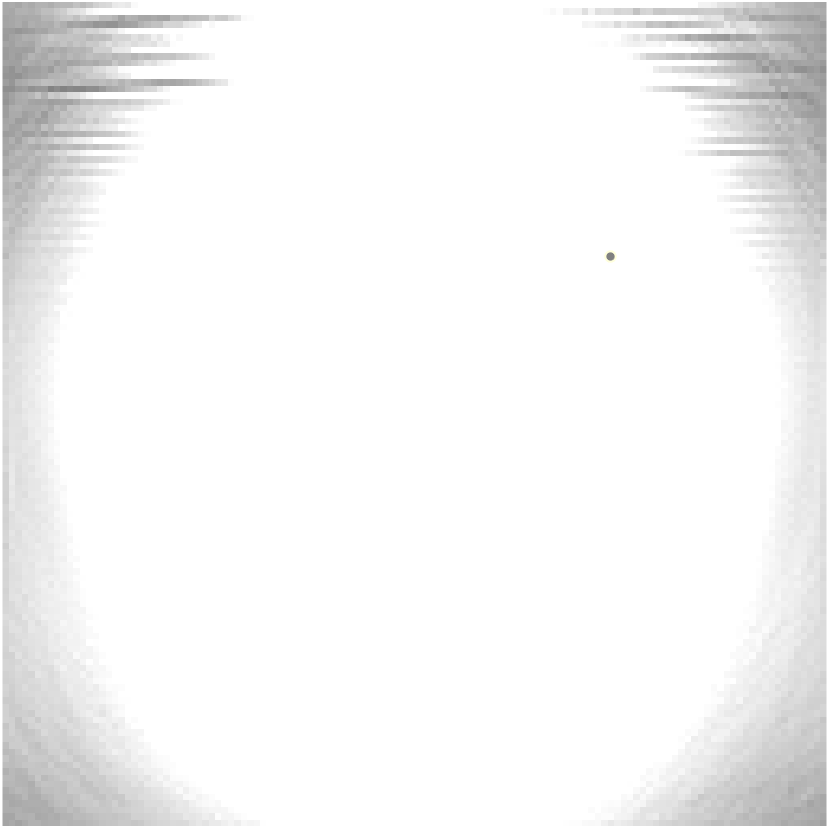}}\,
	\subfloat[\scriptsize FISTA 1]{\label{fig:x_FB_128}\includegraphics[width=0.19\textwidth]{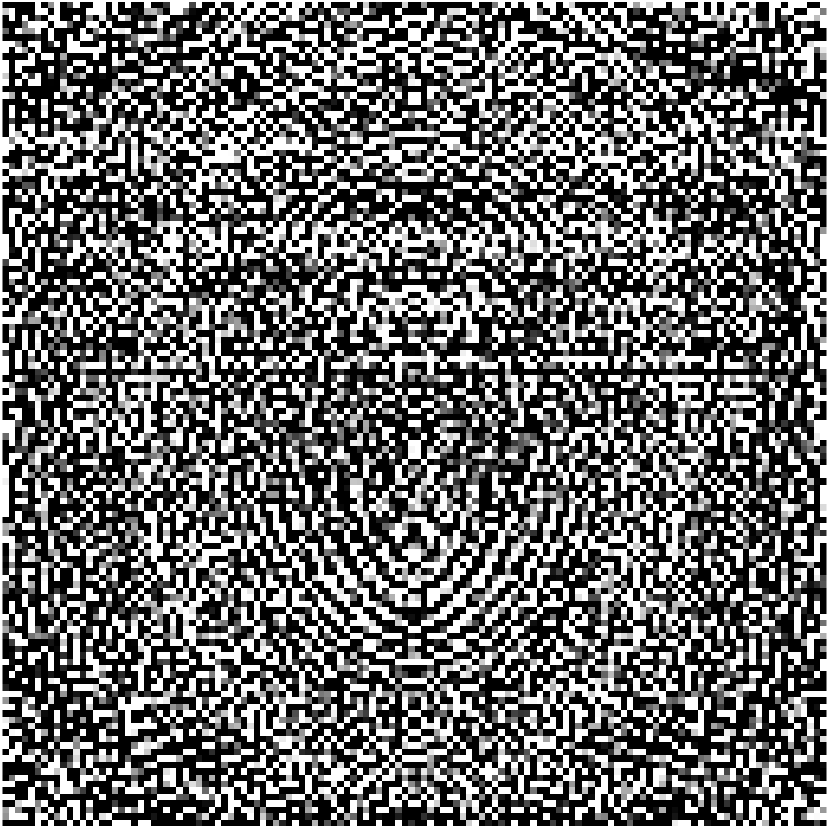}}\,
	\subfloat[\scriptsize FISTA 2]{\label{fig:x_FB2_128}\includegraphics[width=0.19\textwidth]{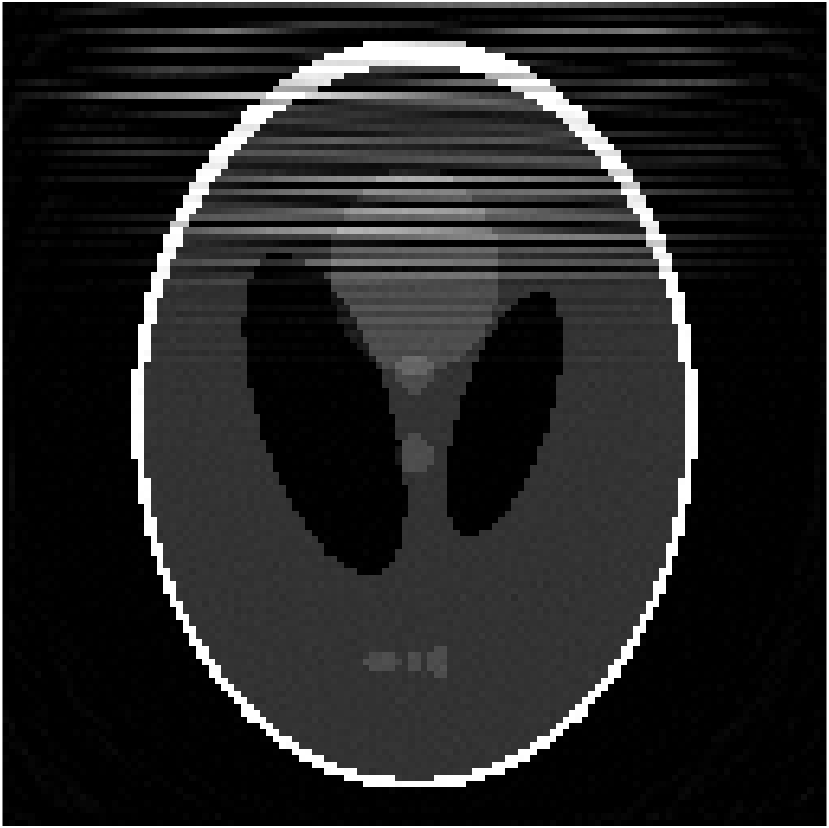}}
	\captionsetup{width=\textwidth} \caption{$m=128$.} 
	\label{fig:n128}
\end{figure}

\section{Appendix: Proof of $r(\delta)$ constant in  Theorem~\ref{t:main}}
	\begin{proof}{Proof}
		Note that $r(\delta)$ defined in \eqref{e:rd} can be written as $r(\delta)=\frac{N(\delta)}{D(\delta)}$, where
		\begin{equation}
			\label{e:rd2}
			\begin{aligned}
				N(\delta)&=\big(ab(1+\alpha \delta) - cd (\rho+\delta)\big)\big(bc(1-\beta \delta)-ad(\mu-\delta)\big)\\
				D(\delta)&=\big(ab(1+\alpha \delta) + cd (\rho+\delta)\big)\big(bc(1-\beta \delta)+ad(\mu-\delta)\big),   
			\end{aligned}
		\end{equation}
		and 
		\begin{equation}
			\label{e:defabcd}
			\begin{aligned}
				a=\sqrt{1+\beta \rho},\:\:
				b= \sqrt{\rho + \mu},\\
				c=\sqrt{1+\alpha \mu},\:\:
				d = \sqrt{\alpha+\beta}.    
			\end{aligned}
		\end{equation}
		Then, by setting $A_1 = (ab- cd \rho ) $, $A_2= (ab\alpha - cd)$, $A_3 =(bc-ad\mu) $, $A_4 =  (bc\beta- ad)$, we have
		\begin{align}\label{eq:app0}
			N(\delta)
			& = \big((ab- cd \rho ) + (ab\alpha - cd)\delta \big) \big((bc-ad\mu)- (bc\beta- ad)\delta\big) \nonumber \\
			& = \big(A_1+ A_2\delta \big) \big(A_3 - A_4\delta\big) \nonumber \\
			& = A_1A_3-(A_1A_4-A_2A_3)\delta-A_2A_4 \delta^2.
			%& = (ab-cd\rho)(cd-ad\mu) + \big( (ab-cd\rho)(ad-cb\beta)-(cd-ab\alpha)(cb-ad\mu) \big) \delta \\
			% & \quad - (cd-ab\alpha)(ad-cb\beta)\delta^2.
		\end{align}
		Moreover, it follows from \eqref{e:defabcd} that
		\begin{align}\label{eq:app1}
			A_1 A_3 &= (ab- cd \rho )(bc-ad\mu) \nonumber \\
			&= ab^2c-a^2bd\mu-c^2bd\rho+acd^2\rho\mu \nonumber\\
			&=ac(b^2+d^2\rho\mu)-bd(a^2\mu+c^2\rho) \nonumber\\
			&=ac(\rho+\mu+(\alpha+\beta)\rho\mu)-bd((1+\beta\rho)\mu+(1+\alpha\mu)\rho) \nonumber\\
			%&=ac(\rho+\mu)-(1+\beta\rho)bd\mu-(1-\alpha\mu)bd\rho+(\alpha+\beta)ac\rho\mu\nonumber \\
			&=(ac-bd)((1+\beta\rho)\mu+(1+\alpha\mu)\rho)
		\end{align}
		and
		\begin{align}\label{eq:app2}
			A_2A_4 &= (ab\alpha - cd)(bc\beta- ad)\nonumber\\
			&=ab^2c\alpha\beta -a^2bd\alpha-bc^2d\beta+acd^2\nonumber\\
			&=ac(b^2\alpha\beta+d^2)-bd(a^2\alpha+c^2\beta)\nonumber\\
			&=ac((\rho+\mu)\alpha\beta+\alpha+\beta)-bd((1+\beta\rho)\alpha+(1+\alpha\mu)\beta)\nonumber\\
			%&=ac\alpha\beta(\rho+\mu) -bd\alpha(1+\beta\rho)-bd\beta(1+\alpha\mu)+ac(\alpha+\beta)\nonumber\\
			&=(ac-bd)((1+\beta\rho)\alpha+(1+\alpha\mu)\beta).
		\end{align}
		Hence, by defining $B_1=((1+\beta\rho)\mu+(1+\alpha\mu)\rho)$ and $B_2 = ((1+\beta\rho)\alpha+(1+\alpha\mu)\beta)$ it follows from \eqref{eq:app0}
		, \eqref{eq:app1}, and  \eqref{eq:app2} that
		\begin{align}\label{eq:app3}
			N(\delta)
			&= (ac-bd)\left( B_1- \frac{(B_1A_4-B_2A_3)}{A_3A_4} \delta-B_2\delta^2\right),
		\end{align}
		where $A_3A_4\neq 0$ because $\beta\mu<1$.
		Analogously, we derive
		\begin{align}\label{eq:app4}
			D(\delta)
			&= (ac+bd)\left( B_1- \frac{(B_1A_4-B_2A_3)}{A_3A_4} \delta-B_2\delta^2\right).
		\end{align}
		Finally, from \eqref{eq:app3} and \eqref{eq:app4} we conclude
		\begin{equation*}
			r(\delta) = \frac{N(\delta)}{D(\delta)}=\frac{ac-bd}{ac+bd} = \frac{\sqrt{(1+\beta\rho)(1+\alpha\mu)}-\sqrt{(\alpha+\beta)(\rho+\mu)}}{\sqrt{(1+\beta\rho)(1+\alpha\mu)}+\sqrt{(\alpha+\beta)(\rho+\mu)}}
		\end{equation*}
		and the result follows.     
	\end{proof}

\section{Acknowledgment}
The work of Luis M. Brice\~no-Arias is supported by  Centro de 
	Modelamiento Matem\'atico (CMM), FB210005, BASAL fund for 
	centers of excellence,  and FONDECYT 
	1230257 from ANID-Chile. The work of Fernando Roldán is supported by ANID-Chile under grant FONDECYT Iniciación 11250164.

%References

%\bibliographystyle{spmpsci.bst}
%\bibliography{references}

\end{document}